\documentclass[11pt,reqno]{amsart}

\usepackage{amsfonts}
\usepackage{amssymb}
\usepackage{hyperref}
\usepackage{graphicx}

\newtheorem{theorem}{Theorem}[section]
\newtheorem{lemma}[theorem]{Lemma}

\newenvironment{definition}{\medskip\noindent \textbf{Definition.}\rm}{\smallskip}

\newenvironment{remark}{\medskip\noindent \textbf{Remark.}\rm}{\smallskip}
\newenvironment{acknowledgement}{\medskip\noindent \textbf{Acknowledgement.}\rm}{\smallskip}
\numberwithin{equation}{section}

\thanks{The first author is supported in part
by SFB 701 of the German Research Council.  The second author is supported in part by
NSF grant DMS-1161622}


\def\RM{\rm}
\def\func#1{\mathop{\mathrm{#1}}\nolimits}
\def\limfunc#1{\mathop{\mathrm{#1}}}

\begin{document}
\title[Estimates of solutions]{Pointwise estimates of solutions to
semilinear elliptic equations and inequalities}
\author{Alexander Grigor'yan}
\address{Department of Mathematics, University of Bielefeld, 33501
Bielefeld, Germany}
\email{grigor@math.uni-bielefeld.de}
\author{Igor Verbitsky}
\address{Department of Mathematics, University of Missouri, Columbia, MO
65211, USA}
\email{verbitskyi@missouri.edu}
\date{\today }

\begin{abstract}
We obtain sharp pointwise estimates for positive solutions to the equation $%
-Lu+Vu^{q}=f$, where $L$ is an elliptic operator in divergence form, $q\in 
\mathbb{R}\setminus \{0\}$, $f\geq 0$ and $V$ is a function that may change
sign, in a domain $\Omega $ in $\mathbb{R}^{n}$, or in a weighted Riemannian
manifold.
\end{abstract}

\maketitle
\tableofcontents


\section{Introduction}

Consider the following elliptic differential equation%
\begin{equation}
-Lu+V\left( x\right) u^{q}=f  \label{uuq}
\end{equation}%
in an open connected set $\Omega \subseteq \mathbb{R}^{n}$, where $q$ is a
non-zero real number, 
\begin{equation}
L=\sum_{i,j=1}^{n}\partial _{x_{i}}\left( a_{ij}\left( x\right) \partial
_{x_{j}}\right)  \label{L}
\end{equation}%
is a divergence form elliptic operator with smooth coefficients $%
a_{ij}=a_{ji}$, $V$ and $f$ are continuous functions in $\Omega $, and $%
f\geq 0$, $f\not\equiv 0$. Note that $V\left( x\right) $ can be signed and
we do not impose any explicit boundary condition on $V$.

Assuming that $u$ is a nonnegative (or positive in the case $q<0$) solution,
our purpose is to obtain pointwise estimates of $u$ in terms of the function 
$h$ that is the minimal positive solution in $\Omega $ of the equation $%
-Lh=f.$ It is not obvious at all, that $u$ should satisfy any bound via $h$,
but nevertheless the following is true.

Assume that the Dirichlet Green function of $L$ in $\Omega $ exists and
denote it by $G^{\Omega }\left( x,y\right) $. Set%
\begin{equation*}
h\left( x\right) =\int_{\Omega }G^{\Omega }\left( x,y\right) f\left(
y\right) dy,
\end{equation*}%
and assume that $h\left( x\right) <\infty $ for all $x\in \Omega $ (note
also that $h\left( x\right) >0$ in $\Omega $), and that the integral%
\begin{equation}
\int_{\Omega }G^{\Omega }\left( x,y\right) h^{q}\left( y\right) V\left(
y\right) dy  \label{Gh}
\end{equation}%
is well-defined. Our main Theorem \ref{T1} states that the following
estimates hold for all $x\in \Omega $.

\begin{itemize}
\item[$\left( i\right) $] If $q=1$ then%
\begin{equation}
u\left( x\right) \geq h\left( x\right) \exp \left( -\frac{1}{h(x)}%
\int_{\Omega }G^{\Omega }\left( x,y\right) h\left( y\right) V\left( y\right)
dy\right) .  \label{uh}
\end{equation}

\item[$\left( ii\right) $] If $q>1$ then 
\begin{equation}
u(x)\geq \frac{h(x)}{\left[ 1+(q-1)\frac{1}{h(x)}\int_{\Omega }G^{\Omega
}\left( x,y\right) h^{q}\left( y\right) V\left( y\right) dy\right] ^{\frac{1%
}{q-1}}},  \label{uhq>1}
\end{equation}%
where the expression in square brackets is necessarily positive, that is,%
\begin{equation}
-(q-1)G^{\Omega }(h^{q}V)(x)<h(x) .  \label{condq>1}
\end{equation}

\item[$\left( iii\right) $] If $0<q<1$ then 
\begin{equation}
u(x)\geq h(x)\left[ 1-(1-q)\frac{1}{h(x)}\int_{\Omega ^{+}}G^{\Omega }\left(
x,y\right) h^{q}\left( y\right) V\left( y\right) dy\right] _{+}^{\frac{1}{1-q%
}},  \label{uhq<1}
\end{equation}%
where 
\begin{equation*}
\Omega ^{+}=\{x\in \Omega :\,\,u(x)>0\}.
\end{equation*}%
In this case we assume that the integral in (\ref{uhq<1}) is well-defined
instead of (\ref{Gh}).

\item[$\left( iv\right) $] If $q<0$, $u>0$ in $\Omega $, and in addition $%
u\left( y\right) \rightarrow 0$ as $y\rightarrow \partial \Omega $ or $%
\left\vert y\right\vert \rightarrow \infty ,$ then (\ref{condq>1}) holds and 
\begin{equation}
u(x)\leq h(x)\left[ 1-(1-q)\frac{1}{h(x)}\int_{\Omega }G^{\Omega }\left(
x,y\right) h^{q}\left( y\right) V\left( y\right) dy\right] ^{\frac{1}{1-q}}.
\label{uhq<0}
\end{equation}
\end{itemize}

Let us emphasize that in the case $\left( iv\right) $ we obtain an upper
bound for $u$ in contrast to the lower bound in the cases $\left( i\right) $-%
$\left( iii\right) $.

In fact, Theorem \ref{T1} holds in much higher generality, when $\Omega $ is
any open subset of any weighted Riemannian manifold, $L$ is the associated
weighted Laplace operator, and equation (\ref{uuq}) can be replaced by an
inequality.

Equation (\ref{uuq}) and its generalizations have attracted attention of
many authors, investigating various aspects from the existence of positive
solutions to pointwise estimates (see, for example, \cite{Ag}, \cite{AS}, 
\cite{BO}, \cite{CZ}, \cite{JMV}, \cite{H}, \cite{HN}, \cite{Mu}, \cite{MV}, 
\cite{P}, \cite{V}, etc). There is no possibility to give a detailed
overview of the literature on this subject, which would have required a full
size survey. We restrict our attention here to those earlier results that
are most closely related to ours.

In the case $q=1$ estimate (\ref{uh}) was known before and is included here
for the sake of completeness. For $V\geq 0$ (\ref{uh}) was proved by Hansen
and Ma \cite[Prop. 1.9]{HM} using the tools of potential theory (see also 
\cite{GH}). For $V\leq 0$ in domains $\Omega $ with boundary Harnack
principle estimate (\ref{uh}) as well a matching upper estimate for $u$ were
obtained in \cite{FV}, \cite{FNV} using a completely different method (but
without sharp constants).

For a general signed $V$ in a relatively compact $\Omega $ estimate (\ref{uh}%
) can be obtained using the Feynman-Kac formula for Brownian motion and
Jensen's inequality. This type of argument was implicit in \cite{AS}, \cite%
{CZ}, \cite[Prop. 2.5]{HZ}. In the form (\ref{uh}) it was stated in \cite%
{GH1}. However, neither the Feynman-Kac formula nor any of the cited above
previous methods allows to treat the nonlinear case $q\neq 1.$

In the case $q>1$ and $V\leq 0$ Kalton and the second author obtained in 
\cite{KV} the necessary condition (\ref{condq>1}), although without a sharp
constant, and gave also a sufficient condition 
\begin{equation}
-G^{\Omega }(h^{q}V)(x)\leq \left( 1-\frac{1}{q}\right) ^{q}\frac{1}{q-1}h(x)
\label{scond}
\end{equation}%
for the existence of a positive solution. Moreover, under (\ref{scond}) they
obtained a two-sided estimate $u\simeq h$ for the minimal positive solution $%
u$ of (\ref{uuq}) in any domain $\Omega $ with the boundary Harnack
principle (the sign $\simeq $ means that the ratio of both sides is bounded
from above and below by positive constants).

In the case $q>1$, $V\leq 0$, and $L=\Delta $, Brezis and Cabr\'{e} \cite{BC}
obtained the sharp necessary condition (\ref{condq>1}) for the existence of
a positive solution in an arbitrary bounded domain $\Omega \subset \mathbb{R}%
^{n}, $ as well as the estimate $u\simeq h$ under (\ref{scond}). The proof
of the necessary condition (\ref{condq>1}) in \cite[Lemma 5.3]{BC} is based
on a direct computation using the explicit form $\Delta
=\sum_{i=1}^{n}\partial _{x_{i}}^{2}$ of the Laplace operator. A much more
expanded version of this computation will appear in our proof in Section \ref%
{SecAux} below.

The case $q>1$, $V\equiv 1$, $f\equiv 0$ has been extensively studied, and
we do not touch it here; we refer the reader to \cite{Dy} and \cite{MV} as
well as to the references therein.

In the case $0<q<1$, $V\leq 0,$ and $L=\Delta $, Brezis and Kamin \cite{BK}
obtained necessary and sufficient conditions for the existence of a bounded,
positive solution of (\ref{uuq}) in $\mathbb{R}^{n}$ and obtained certain
pointwise bounds. Their lower bound is covered by our Theorem \ref{main-thm3}
below (see also \cite{DV1}, \cite{DV2}).

In the case $q<0$ \cite{DGR}, \cite{GhR} obtained a sharp sufficient
condition for the existence of a positive solution of (\ref{uuq}) in the
specific case where $V\left( x\right) $ depends only on the distance from $x$
to $\partial \Omega $ and has a constant sign.

In the present paper we give a unified approach for treating all the values
of $q\in \mathbb{R}\setminus \left\{ 0\right\} $, a general signed potential 
$V$, and a general divergence form operator $L$, not only in arbitrary
domains of $\mathbb{R}^{n}$, but also on an arbitrary Riemannian manifold.
Our estimates $\left( i\right) $-$\left( iv\right) $ are new in this
generality. In many cases these estimates happen to be sharp as one can see
in examples in Section \ref{examples}.

Let us briefly describe the idea of our proof. Assume for simplicity $%
L=\Delta $. Let $\left\{ \Omega _{k}\right\} _{k=1}^{\infty }$ be an
exhaustion of $\Omega $ by relatively compact open sets $\Omega _{k}\subset
\Omega $ with smooth boundaries. We obtain first appropriate estimates for $%
u $ in each $\Omega _{k}$ and then pass to the limit as $k\rightarrow \infty 
$. Define in $\Omega _{k}$ a new function $h$ as the solution of the
following boundary value problem%
\begin{equation*}
\left\{ 
\begin{array}{ll}
-\Delta h=f, & \text{in }\Omega _{k}, \\ 
h=u, & \text{on }\partial \Omega _{k}.%
\end{array}%
\right.
\end{equation*}%
The following argument is used in the proof of Theorem \ref{T2} that treats (%
\ref{uuq}) in relatively compact domains with the Dirichlet boundary
condition. Assume first that $h\equiv 1$ (and then $f=0$ in $\Omega _{k}$).
Fix a $C^{2} $ function $\phi $ on $\mathbb{R}$ (or on an interval in $%
\mathbb{R}$) with $\phi ^{\prime }>0$ and consider the substitution 
\begin{equation*}
v=\phi ^{-1}\left( u\right) .
\end{equation*}%
By the chain rule we have%
\begin{equation*}
\Delta u=\Delta \phi \left( v\right) =\phi ^{\prime }(v)\Delta v+\phi
^{\prime \prime }(v)|\nabla v|^{2},
\end{equation*}%
which implies%
\begin{eqnarray}
-\Delta v+V\frac{\phi (v)^{q}}{\phi ^{\prime }(v)} &=&-\frac{\Delta u-\phi
^{\prime \prime }\left\vert \nabla v\right\vert ^{2}}{\phi ^{\prime }}+V%
\frac{\phi (v)^{q}}{\phi ^{\prime }(v)}  \notag \\
&=&-\frac{V\phi \left( v\right) ^{q}-\phi ^{\prime \prime }\left\vert \nabla
v\right\vert ^{2}}{\phi ^{\prime }}+V\frac{\phi (v)^{q}}{\phi ^{\prime }(v)}
\notag \\
&=&\frac{\phi ^{\prime \prime }}{\phi ^{\prime }}\left\vert \nabla
v\right\vert ^{2}.  \label{3}
\end{eqnarray}%
Now we choose $\phi $ to solve the initial value problem%
\begin{equation*}
\phi ^{\prime }\left( s\right) =\phi ^{q}\left( s\right) ,\ \ \phi \left(
0\right) =1,
\end{equation*}%
and obtain:%
\begin{equation*}
\phi \left( s\right) =\left\{ 
\begin{array}{ll}
e^{s}, & q=1, \\ 
\lbrack (1-q)s+1]^{\frac{1}{1-q}}, & q\neq 1,%
\end{array}%
\right.
\end{equation*}%
in the appropriate domains. In the case $q>0$ the function $\phi $ is
convex, and we obtain from (\ref{3})%
\begin{equation}
-\Delta v+V\geq 0.  \label{vv}
\end{equation}%
Since on $\partial \Omega _{k}$ we have $v=\phi ^{-1}\left( u\right) =\phi
^{-1}\left( 1\right) =0$, we obtain from (\ref{vv}) by the maximum principle
that%
\begin{equation*}
v\left( x\right) \geq -\int_{\Omega _{k}}G^{\Omega _{k}}\left( x,y\right)
V\left( y\right) dy.
\end{equation*}%
Applying $\phi $ to both sides of this inequality gives an appropriate
inequality for $u=\phi \left( v\right) $ in $\Omega _{k}.$

In the case $q<0$ the function $\phi $ is concave, which leads to the
opposite inequality for $v$ and, hence, for $u$.

In the case of a general function $h$, consider a so-called $h$-transform
(or Doob's transform \cite{Do}) of $\Delta $:%
\begin{equation*}
\Delta ^{h}=\frac{1}{h}\circ \Delta \circ h=\frac{1}{h^{2}}\func{div}\left(
h^{2}\nabla \right) +\frac{\Delta h}{h},
\end{equation*}%
and the function $\widetilde{u}=\frac{u}{h}$. Then $\widetilde{u}$ solves
the equation%
\begin{equation*}
-\Delta ^{h}\widetilde{u}+h^{q-1}V\widetilde{u}^{q}=-\frac{\Delta h}{h}
\end{equation*}%
with the boundary value $\widetilde{u}=1$ on $\partial \Omega _{k}$.
Effectively the $h$-transform provides a reduction to the previous case, but
for the operator $\Delta ^{h}$ in place of $\Delta $. The part $\frac{1}{%
h^{2}}\func{div}\left( h^{2}\nabla \right) $ of this operator is a weighted
Laplace operator, for which the same computation (\ref{3}) using the chain
rule works as for $\Delta $. The part $\frac{\Delta h}{h}$ gives in the end
an additional term%
\begin{equation*}
\frac{\Delta h}{h}\left( \frac{\phi (v)-1}{\phi ^{\prime }(v)}-v\right)
\end{equation*}%
on the right hand side of (\ref{vv}) (cf. Lemma \ref{ident}). In the case $%
q>1$ we obtain by the convexity of $\phi $ that the expression in
parentheses is non-positive. Since $\Delta h=-f\leq 0$, the above term is
non-negative which allows us to use the same argument as above. In the case $%
q<1$ this term is non-positive, which gives again a correct sign in the
corresponding inequality.

The actual proof goes a bit differently as we have to overcome one more
difficulty -- a possibility of $h$ vanishing on the boundary, which we have
ignored in the above sketch (see Sections \ref{SecProofB}, \ref{SecT1}).

The above argument allows a version that treats the case $f=0$ in (\ref{uuq}%
) -- see Theorem \ref{T3}.

In Theorem \ref{C4} we provide complementary results: sufficient conditions
for the existence of a positive solution $u$ and two-sided estimates of $u$.
Finally, Theorem \ref{T4} is an abstract version of Theorem \ref{C4} for
solutions of integral equations.

The structure of the paper is as follows. In Section \ref{SecW} we briefly
describe the notion of the weighted manifold and the associated Laplace
operator. In Section \ref{SecMain} we state our main results: Theorems \ref%
{T1}, \ref{T2}, \ref{T3}, \ref{C4} and \ref{T4}. In Section \ref{SecAux} we
prove some Lemmas, in particular containing the aforementioned computation (%
\ref{3}) in the general case. In Section \ref{SecProofB}-\ref{SecProofE} we
prove the above mentioned theorems. In Section \ref{examples} we give some
examples.

\begin{acknowledgement}
The authors are grateful to Alexander Bendikov, Haim Brezis, and Wolfhard
Hansen for stimulating conversations on the subject of this paper.

The second author would like to thank the Mathematics Department at
Bielefeld University for its hospitality and support.
\end{acknowledgement}

\section{Weighted manifolds}

\label{SecW}Let $M$ be a smooth Riemannian manifold with the Riemannian
metric tensor $g=\left( g_{ij}\right) $. The associated Laplace-Beltrami
operator $\mathcal{L}_{0}$ acts on $C^{2}$ functions $u$ on $M$ and is given
in any chart $x_{1},...,x_{n}$ by the formula%
\begin{equation*}
\mathcal{L}_{0}u=\frac{1}{\sqrt{\det g}}\sum_{i,j=1}^{n}\partial
_{x_{i}}\left( \sqrt{\det g}g^{ij}\partial _{x_{j}}u\right)
\end{equation*}%
where $\det g$ is the determinant of the matrix $g=\left( g_{ij}\right) $,
and $\left( g^{ij}\right) $ is the inverse matrix of $\left( g_{ij}\right) $%
. The Riemannian measure $m_{0}$ is given in the same chart by%
\begin{equation*}
dm_{0}=\sqrt{\det g}dx_{1}...dx_{n},
\end{equation*}%
so that $\mathcal{L}_{0}$ is symmetric with respect to $m_{0}.$ Using the
gradient operator $\nabla $ defined by 
\begin{equation*}
\left( \nabla u\right) ^{i}=\sum_{j=1}^{n}g^{ij}\partial _{x_{j}}u
\end{equation*}%
and the divergence $\func{div}$ on vector fields $F^{i}$%
\begin{equation*}
\func{div}F=\frac{1}{\sqrt{\det g}}\sum_{i=1}^{n}\partial _{x_{i}}\left( 
\sqrt{\det g}F^{i}\right) ,
\end{equation*}%
one represents $\mathcal{L}_{0}$ in the form 
\begin{equation*}
\mathcal{L}_{0}=\func{div}\circ \nabla .
\end{equation*}

Let $\omega $ be a smooth positive function on $M$ and consider the measure $%
m$ on $M$ given by 
\begin{equation*}
dm=\omega dm_{0}.
\end{equation*}%
The couple $\left( M,m\right) $ is called a weighted manifold or a manifold
with density, and $\omega $ in this context is called a weight. The
following operator $\mathcal{L}$ 
\begin{equation}
\mathcal{L}u:=\frac{1}{\omega }\func{div}\left( \omega \nabla u\right) =%
\frac{1}{\omega \sqrt{\det g}}\sum_{i,j=1}^{n}\partial _{x_{i}}\left( \omega 
\sqrt{\det g}g^{ij}\partial _{x_{j}}u\right)  \label{Lw}
\end{equation}%
acting on $C^{2}$ functions $u$ on $M$, is called the (weighted) Laplace
operator of $\left( M,m\right) $. It is easy to see that $\mathcal{L}$ is
symmetric with respect to measure $m$.

Of course, for $\omega =1$ we have $\mathcal{L}=\mathcal{L}_{0}$. For a
general weight $\omega $, define the weighted divergence by 
\begin{equation*}
\func{div}_{\omega }=\frac{1}{\omega }\circ \func{div}\circ \omega
\end{equation*}%
and obtain 
\begin{equation*}
\mathcal{L}=\func{div}_{\omega }\circ \nabla .
\end{equation*}%
Note that $\nabla $ remains the Riemannian gradient and does not depend on
the weight $\omega $.

It is easy to show that the weighted Laplace operator $\mathcal{L}$
satisfies the same product and chain rules as the classical Laplace operator
(cf. \cite[Section 3.6]{G1}). Namely, for two $C^{2}$ functions $u,v $ on $M$
we have%
\begin{equation}
\mathcal{L}\left( uv\right) =u\mathcal{L}v+2\langle \nabla u,\nabla v\rangle
+v\mathcal{L}u  \label{product}
\end{equation}%
where $\langle \nabla u,\nabla v\rangle $ is the inner product of the
Riemannian gradients, which is independent of the weight $\omega $. Also,
for any $C^{2}$ function $\phi $ defined on $u\left( M\right) $ we have%
\begin{equation}
\mathcal{L}\phi \left( u\right) =\phi ^{\prime }\left( u\right) \mathcal{L}%
u+\phi ^{\prime \prime }\left( u\right) \left\vert \nabla u\right\vert ^{2}.
\label{chain}
\end{equation}

As an example, consider in an open set $\Omega \subseteq \mathbb{R}^{n}$ the
following operator%
\begin{equation}
Lu=b\left( x\right) \sum_{i,j=1}^{n}\partial _{x_{i}}\left( a_{ij}\left(
x\right) \partial _{x_{j}}u\right) ,  \label{Lba}
\end{equation}%
where $b,a_{ij}$ are smooth functions, $b>0$ and $a_{ij}=a_{ji}$. Assume
that $L$ is elliptic, that is, the matrix $\left( a_{ij}\left( x\right)
\right) $ is positive definite for any $x$ (the uniform ellipticity is not
assumed). Then $L$ coincides with the weighted Laplace operator $\mathcal{L}$
of $\mathbb{R}^{n}$ with the Riemannian metric $g$ and weight $\omega $
given by%
\begin{equation*}
\left( g^{ij}\right) =b\left( a_{ij}\right) ,\ \ \ \ \omega =b^{\frac{n}{2}%
-1}\sqrt{\det a},
\end{equation*}%
where $a=\left( a_{i_{j}}\right) $. Indeed, it follows that%
\begin{equation*}
\det g=\det \left( g_{ij}\right) =\frac{1}{b^{n}\det a},
\end{equation*}%
and substitution into (\ref{Lw}) yields 
\begin{eqnarray*}
\mathcal{L}u &=&\frac{\sqrt{b^{n}\det a}}{b^{\frac{n}{2}-1}\sqrt{\det a}}%
\sum_{i,j=1}^{n}\partial _{x_{i}}\left( b^{\frac{n}{2}-1}\sqrt{\det a}\frac{1%
}{\sqrt{b^{n}\det a}}ba^{ij}\partial _{x_{j}}u\right) \\
&=&b\sum_{i,j=1}^{n}\partial _{x_{i}}\left( a_{ij}\left( x\right) \partial
_{x_{j}}u\right) =Lu.
\end{eqnarray*}%
The measure $m$ associated with $\mathcal{L}$ is given by%
\begin{equation}
dm=\omega \sqrt{\det g}=b^{\frac{n}{2}-1}\sqrt{\det a}\frac{1}{\sqrt{%
b^{n}\det a}}=\frac{1}{b}dx,  \label{dm}
\end{equation}%
where $dx$ is Lebesgue measure.

Therefore, all the results that we obtain for a general weighted manifold $%
\left( M,m\right) $, apply to the operator (\ref{Lba}) in a domain of $%
\mathbb{R}^{n}$ with the measure $m$ from (\ref{dm}). In particular, if $%
b\equiv 1$ as was assumed in the Introduction, then $L$ is given by (\ref{L}%
) and $m$ is Lebesgue measure.

\section{Statements of the main results}

\label{SecMain}For any open connected set $\Omega \subseteq M$ denote by $%
G^{\Omega }\left( x,y\right) $ the infimum of all positive fundamental
solutions of $\mathcal{L}$ in $\Omega $. The following dichotomy is true:
either $G^{\Omega }\left( x,y\right) \equiv \infty $ or $G^{\Omega }\left(
x,y\right) <\infty $ for all $x\neq y$. In the latter case we say that $%
G^{\Omega }$ is finite. If $G^{\Omega }$ is finite then $G^{\Omega }$ is the
symmetric positive Green function of $\mathcal{L}$ in $\Omega $ (see \cite{G}
and \cite[Ch.13]{G1}). If $\Omega $ is relatively compact then $G^{\Omega }$
is finite and satisfies the Dirichlet boundary condition on the regular part
of $\partial \Omega $.

If $G^{\Omega }$ is finite then, for any function $f\in L_{loc}^{1}\left(
\Omega, m\right) $, set%
\begin{equation*}
G^{\Omega }f\left( x\right) =\int_{\Omega }G^{\Omega }\left( x,y\right)
f\left( y\right) dm\left( y\right) ,
\end{equation*}%
where in the case $f\geq 0$ the integral is understood in the sense of
Lebesgue; for a signed $f$ the integral is understood as follows:%
\begin{equation*}
G^{\Omega }f\left( x\right) =G^{\Omega }f_{+}\left( x\right) -G^{\Omega
}f_{-}\left( x\right)
\end{equation*}%
(where $f_{+}=\max \left( f,0\right) $ and $f_{-}=\max \left( -f,0\right) $%
), assuming that at least one of the values $G^{\Omega }f_{+}\left( x\right) 
$, $G^{\Omega }f_{-}\left( x\right) $ is finite. In this case we say that $%
G^{\Omega }f\left( x\right) $ is well-defined.

Note that if $f\geq 0$ in $\Omega $ and $f>0$ on a set of positive measure
then $G^{\Omega }f>0$ in $\Omega $.

If $\Omega $ is relatively compact then $G^{\Omega }\left( x,\cdot \right)
\in L^{1}\left( \Omega \right) $, which implies that $G^{\Omega }f$ is
finite for any $f\in L^{\infty }\left( \Omega \right) $. For arbitrary $%
\Omega $ it is still true that $G^{\Omega }(x,\cdot )\in L_{\mathrm{loc}%
}^{1}(\Omega )$ for every $x\in \Omega $.

Denote by $\partial _{\infty }M$ the infinity point of the one-point
compactification of $M$ (see for example \cite[Sec. 5.4.3]{G1}). For any
open subset $\Omega \subseteq M$ denote by $\partial _{\infty }\Omega $ the
union of $\partial \Omega $ and $\partial _{\infty }M$, if $\Omega $ is not
relatively compact, and set $\partial _{\infty }\Omega =\partial \Omega $ if 
$\Omega $ is relatively compact.

\begin{definition}
For a function $u$ defined in $\Omega \subseteq M$ let us write%
\begin{equation}
\lim_{y\rightarrow \partial _{\infty }\Omega }u\left( y\right) =0,
\label{lu}
\end{equation}%
if $\lim_{k\rightarrow \infty }u\left( y_{k}\right) =0$ for any sequence $%
\left\{ y_{k}\right\} $ in $\Omega $ that converges to a point of $\partial
_{\infty }\Omega $; the latter means, that either $\left\{ y_{k}\right\} $
converges to a point on $\partial \Omega $ or diverges to $\partial _{\infty
}M$. In the same way we understand similar equalities and inequalities
involving $\limsup $ and $\liminf .$
\end{definition}

For example, if $\Omega $ is relatively compact, then (\ref{lu}) means that $%
\lim u\left( y_{k}\right) =0$ for any sequence $\left\{ y_{k}\right\} $
converging to a point on $\partial \Omega $. If $\Omega =M$ then $\partial
\Omega =\emptyset $ and (\ref{lu}) means that $\lim u\left( y_{k}\right) =0$
for any sequence $y_{k}\rightarrow \partial _{\infty }M$, that is, for any
sequence $\left\{ y_{k}\right\} $ that leaves any compact subset of $M$. In
particular, for $M=\mathbb{R}^{n}$ (\ref{lu}) is equivalent to $u\left(
y\right) \rightarrow 0$ as $\left\vert y\right\vert \rightarrow 0$.

We will use the notation 
\begin{equation*}
\chi _{u}\left( x\right) =\left\{ 
\begin{array}{ll}
1, & u\left( x\right) >0, \\ 
0, & u\left( x\right) \leq 0.%
\end{array}%
\right.
\end{equation*}

\begin{theorem}
\label{main-thm2}\label{T1}Let $M$ be an arbitrary weighted manifold, and
let $\Omega \subseteq M$ be a connected open subset of $M$ with a finite
Green function $G^{\Omega }$. Suppose $V,f\in C(\Omega )$ and assume $f\geq
0,$ $f\not\equiv 0$ in $\Omega $. Let $u\in C^{2}(\Omega )$ satisfy 
\begin{equation}
\text{in the case }q>0:\ \text{\ }-\mathcal{L}u+Vu^{q}\geq f\ \ \text{in}%
\,\,\Omega ,\ \ u\geq 0,  \label{semilinear2A}
\end{equation}%
or,%
\begin{equation}
\text{in the case }q<0:\ \left\{ 
\begin{array}{l}
-\mathcal{L}u+Vu^{q}\leq f\ \ \text{in}\,\,\Omega \text{,}\  \\ 
\lim_{y\rightarrow \partial _{\infty }\Omega }u\left( y\right) =0,%
\end{array}%
\right. \text{\ }u>0.  \label{semilinear2B}
\end{equation}%
Set $h=G^{\Omega }f$ and assume that $h<\infty $ in $\Omega $. Assume also
that $G^{\Omega }(h^{q}V)(x)$ (respectively $G^{\Omega }(\chi _{u}h^{q}V)(x)$
in the case $0<q<1$) is well-defined for all $x\in \Omega $. Then the
following statements hold for all $x\in \Omega $.

$\left( i\right) $ If $q=1$, then 
\begin{equation}
u(x)\geq h(x)e^{-\frac{1}{h(x)}G^{\Omega }(hV)(x)}.  \label{main-exp}
\end{equation}

$\left( ii\right) $ If $q>1$, then necessarily 
\begin{equation}
-(q-1)G^{\Omega }(h^{q}V)(x)<h(x),  \label{main-cond}
\end{equation}%
and the following estimate holds: 
\begin{equation}
u(x)\geq \frac{h(x)}{\left[ 1+(q-1)\frac{G^{\Omega }(h^{q}V)(x)}{h(x)}\right]
^{\frac{1}{q-1}}}.  \label{main1}
\end{equation}

$\left( iii\right) $ If $0<q<1$, then 
\begin{equation}
u(x)\geq h(x)\left[ 1-(1-q)\frac{G^{\Omega }(\chi _{u}h^{q}V)(x)}{h(x)}%
\right] _{+}^{\frac{1}{1-q}}.  \label{main}
\end{equation}

$\left( iv\right) $ If $q<0$ then necessarily \emph{(\ref{main-cond})}
holds, and 
\begin{equation}
u(x)\leq h(x)\left[ 1-(1-q)\frac{G^{\Omega }(h^{q}V)(x)}{h(x)}\right] ^{%
\frac{1}{1-q}}.  \label{main-negative}
\end{equation}
\end{theorem}

Note that the condition $f\not\equiv 0$ implies $h>0$ in $\Omega $. Note
also that without loss of generality the open set $\Omega $ in Theorem \ref%
{T1} can be taken to be $M$. However, we have preferred the present
formulation for the sake of convenience in applications.

\begin{remark}
\RM
In the case $q\geq 1,$ it follows from (\ref{main-exp}) and (\ref{main1})
that the condition%
\begin{equation*}
G^{\Omega }\left( h^{q}V\right) \left( x\right) <+\infty
\end{equation*}%
implies $u\left( x\right) >0$. Moreover, if for some $C>0$ and all $x\in
\Omega $%
\begin{equation*}
G^{\Omega }\left( h^{q}V\right) \left( x\right) \le Ch\left( x\right) ,
\end{equation*}%
then $u\geq ch$ in $\Omega $ with some constant $c=c\left( C\right) >0$.

In the case $0<q<1$ the function $u$ can vanish in $\Omega $, but the
estimate of $u$ cannot depend on the values of $V$ on the set $\left\{
u=0\right\} $. This explains the appearance of the factor $\chi _{u}$ and
the subscript $+$ on the right-hand side of (\ref{main}).

In the case $q<0$, the boundary condition $\lim_{y\rightarrow \partial
_{\infty }\Omega }u\left( y\right) =0$ is needed as without this condition,
for positive $V$, the function $u+C$ would also be a solution to (\ref%
{semilinear2B}) for any $C>0$, so that $u$ could not admit any upper bound.
\end{remark}

\begin{remark}
\RM The lower estimates of Theorem~\ref{T1} $\left( i\right) ,\left(
ii\right) ,\left( iii\right) $ remain valid even if the expression $%
G^{\Omega }(h^{q}V)$ is not well-defined in the above sense, provided it is
understood as follows%
\begin{equation}
G^{\Omega }(h^{q}V)(x):=\liminf_{n\rightarrow \infty }\int_{\Omega
_{n}}G^{\Omega _{n}}(x,y)h^{q}(y)V(y)dy,  \label{improper}
\end{equation}%
where $\left\{ \Omega _{n}\right\} $ is any exhaustion of $\Omega $ by
relatively compact subsets with smooth boundaries. The same is true for the
upper estimate of $\left( iv\right) $ where one can use $\limsup $ in place
of $\liminf $.

In the case $q=1$ and $h=G^{\Omega }f$, this means 
\begin{equation}
G^{\Omega }(hV)(x)=G_{2}^{\Omega }f(x)=\liminf_{n\rightarrow \infty
}\int_{\Omega _{n}}G_{2}^{\Omega _{n}}(x,y)f(y)dy,\quad x\in \Omega ,
\label{improper2}
\end{equation}%
where $G_{2}^{\Omega }$ stands for the second iteration of the Green kernel
with respect to $V(y)dy$: 
\begin{equation}
G_{2}^{\Omega }(x,y)=\int_{\Omega }G^{\Omega }(x,z)G^{\Omega
}(z,y)V(z)dz,\quad x,y\in \Omega .  \label{K2}
\end{equation}%
In some cases $G_{2}^{\Omega }(x,y)$ in (\ref{K2}) can be understood as an
improper integral. (See Example 1 in Section \ref{examples} below.)
\end{remark}

\begin{remark}
\RM Suppose $q>1$ in Theorem~\ref{T1}. The necessary condition (\ref%
{main-cond}) for the existence of a positive solution of (\ref{semilinear2A}%
) in the case $V\leq 0$ was proved in \cite{KV}, without the sharp constant $%
\frac{1}{q-1}$, but for general quasi-metric kernels, including a wide
variety of differential and integral operators. It was also shown in \cite%
{KV} that the stronger condition 
\begin{equation}
-G^{\Omega }(h^{q}V)(x)\leq \left( 1-\frac{1}{q}\right) ^{q}\frac{1}{q-1}%
\,h(x),\quad x\in \Omega ,  \label{sufficient}
\end{equation}%
is sufficient for the existence of a solution $u$ such that 
\begin{equation*}
h\leq u\leq C(q)\,h
\end{equation*}%
Brezis and Cabr\'{e} \cite{BC} subsequently proved the necessity of (\ref%
{main-cond}) with the sharp constant $\frac{1}{q-1}$ in the case of $%
\mathcal{L}=\Delta $ in bounded domains of $\mathbb{R}^{n}$ (see also
Theorem \ref{T4} below).
\end{remark}

In the proof of Theorem \ref{T1}, we use Theorem \ref{T2} below that deals
with relatively compact sets $\Omega \subset M$. Fix a function $h\in
C^{2}\left( \Omega \right) \cap C\left( \overline{\Omega }\right) $ such
that 
\begin{equation}
h>0\text{ \ in\ \ }\Omega \ \ \text{and }-\mathcal{L}h\geq 0\ \ \text{in\ \ }%
\Omega .  \label{h}
\end{equation}%
Consider in $\Omega $ the following boundary value inequalities:%
\begin{equation}
\begin{cases}
-\mathcal{L}u+Vu^{q}\geq -\mathcal{L}h\ \ \ \text{in }\Omega \\ 
u\geq h\quad \text{on }\partial \Omega \\ 
u\geq 0\ \ \ \ \text{in }\Omega%
\end{cases}%
\ \ \ \ \text{in the case }q>0,  \label{semilinearA}
\end{equation}%
and%
\begin{equation}
\begin{cases}
-\mathcal{L}u+Vu^{q}\leq -\mathcal{L}h\ \ \ \text{in }\Omega \quad \\ 
u\leq h\ \ \ \ \text{on }\partial \Omega \\ 
u>0\ \ \ \ \ \text{in }\Omega%
\end{cases}%
\ \ \ \text{in the case }q<0,  \label{semilinearB}
\end{equation}%
where $V\in C(\Omega )$ and $u\in C^{2}(\Omega )\cap C\left( \overline{%
\Omega }\right) $. In the next theorem we compare $u$ and $h$ as follows.

\begin{theorem}
\label{main-thm}\label{T2} Let $\left( M,m\right) $ be an arbitrary weighted
manifold, and let $\Omega \subset M$ be a relatively compact connected open
subset of $M$. Let a function $h\in C^{2}\left( \Omega \right) \cap C\left( 
\overline{\Omega }\right) $ satisfy \emph{(\ref{h})}.

Let $V\in C(\Omega )$ and suppose that $u\in C^{2}(\Omega )\cap C\left( 
\overline{\Omega }\right) $ is a solution to either \emph{(\ref{semilinearA})%
} or \emph{(\ref{semilinearB})}. Assume also that $G^{\Omega }(h^{q}V)(x)$
(respectively $G^{\Omega }(\chi _{u}h^{q}V)(x)$ in the case $0<q<1$) is
well-defined for all $x\in \Omega $. Then statements $\left( i\right) $-$%
\left( iv\right) $ of Theorem \emph{\ref{T1}} hold.
\end{theorem}

\begin{remark}
\RM In the linear case $q=1$, we obtain a simple proof of the well-known
lower estimate of solutions to the Schr\"{o}dinger equation: 
\begin{equation}
u(x)\geq h(x)e^{-\frac{1}{h(x)}G^{\Omega }(hV)(x)},\quad \text{for all }%
\,\,x\in \Omega .  \label{main-exp2}
\end{equation}%
This estimate in the special case $h=1$ is usually deduced via the
Feynman-Kac formalism (see \cite{AS}, \cite{CZ}) using Jensen's inequality.
In the case $V\geq 0$, alternative proofs based on potential theory methods
in a very general setting are given in \cite{G}, \cite{GH}. In the case $%
V\leq 0$, a similar lower estimate and a matching upper estimate (but
without sharp constants) are obtained in \cite{FV}, \cite{FNV} for general
quasi-metric kernels.

An interesting special case is when $h$ is the solution of the Dirichlet
problem: 
\begin{equation}
\begin{cases}
-\mathcal{L}h=1\quad \text{in}\,\,\Omega , \\ 
h=0\quad \text{on}\,\,\partial \Omega .%
\end{cases}
\label{exit-time}
\end{equation}%
In other words, $h(x)=\mathbb{E}_{x}[\tau _{\Omega }]$, where $\tau _{\Omega
}=\inf \{t\!:\,X_{t}\not\in \Omega \}$ is the first exit time from $\Omega $
of the (rescaled) Brownian motion $X_{t}$, and $x\in \Omega $ is a starting
point. For bounded $C^{1,1}$ domains, $h(x)\simeq d_{\Omega }(x)$, where 
\begin{equation}
d_{\Omega }(x)=\limfunc{dist}(x,\partial \Omega ).  \label{dOm}
\end{equation}%
This gives sharp estimates: 
\begin{equation}
u(x)\geq c\,d_{\Omega }(x)\,e^{-\frac{c}{d_{\Omega }(x)}G^{\Omega
}(d_{\Omega }V)(x)},\quad \text{for all }\,\,x\in \Omega .  \label{main-exp1}
\end{equation}%
if $q=1$, as well as the corresponding estimates for other values of $q$.

For bounded Lipschitz domains with sufficiently small Lipschitz constant
(less than $(n-1)^{1/2}$, which is sharp), it is known that (see \cite{AAC}) 
\begin{equation*}
h(x)\simeq \rho (x)=\min (1,G^{\Omega }(x,x_{0})),
\end{equation*}%
where $x_{0}$ is a fixed pole in $\Omega $, and so (\ref{main-exp1}) holds
with $\rho $ in place of $d_{\Omega }$. The corresponding estimates hold for
other values of $q\in \mathbb{R}$ as well.
\end{remark}

Returning again to the case of an arbitrary (not necessarily relatively
compact) $\Omega $, in the next theorem we give estimates of solutions $u$
of (\ref{semilinear2A})-(\ref{semilinear2B}) with $f=0.$ They are applicable
to the so-called gauge ($q=1$), \textquotedblleft large\textquotedblright\
solutions ($q>1$), or \textquotedblleft ground state\textquotedblright\
solutions ($-\infty <q<1$) to the corresponding equations in unbounded
domains in $\mathbb{R}^{n}$ or non-compact manifolds.

\begin{theorem}
\label{main-thm3}\label{T3} Let $M$ be an arbitrary weighted manifold, and
let $\Omega \subseteq M$ be an open connected set with a finite Green
function $G^{\Omega }$. Suppose $V\in C(\Omega )$. Let $u\in C^{2}(\Omega )$
satisfy either the inequality 
\begin{equation}
-\mathcal{L}u+V\,u^{q}\geq 0,\quad u\geq 0\,\,\text{in}\,\,\Omega ,\ \ \text{%
if }q>0,  \label{semilinearC}
\end{equation}%
or 
\begin{equation}
-\mathcal{L}u+V\,u^{q}\leq 0,\quad u>0\,\,\text{in}\,\,\Omega ,\ \text{if }%
q<0.  \label{semilinearD}
\end{equation}%
Assume also that $G^{\Omega } V (x)$ (respectively $G^{\Omega }(\chi _{u}
V)(x)$ in the case $0<q<1$) is well-defined for all $x\in \Omega $. Then the
following statements hold for all $x \in \Omega$.

$(i)$ If $q=1$ and 
\begin{equation}
\liminf_{y\rightarrow \partial _{\infty }\Omega }\,u\left( y\right) \geq 1
\label{u=1-thm3}
\end{equation}%
then 
\begin{equation}
u(x)\geq e^{-\ G^{\Omega }V(x)}.  \label{est1-thm3}
\end{equation}

$(ii)$ If $q>1$ and 
\begin{equation}
\lim_{y\rightarrow \partial _{\infty }\Omega }\,u\left( y\right) =+\infty \,,
\label{u=infty-thm3}
\end{equation}%
then necessarily $G^{\Omega }V(x)>0$, and 
\begin{equation}
u(x)\geq \left[ (q-1)\,G^{\Omega }V(x)\right] ^{-\frac{1}{q-1}}.
\label{est2-thm3}
\end{equation}

$(iii)$ If $0<q<1$, then 
\begin{equation}
u(x)\geq \left[ -(1-q)\,G^{\Omega }(\chi _{u}V)(x)\right] _{+}^{\frac{1}{1-q}%
}.  \label{est3-thm3}
\end{equation}

$(iv)$ If $q<0$, and 
\begin{equation}
\lim_{y\rightarrow \partial _{\infty }\Omega }\,u\left( y\right) =0,
\label{u=0-thm3}
\end{equation}%
then necessarily $G^{\Omega }V(x) < 0$, and 
\begin{equation}
u(x)\leq \left[ -(1-q)\,G^{\Omega }V(x)\right] ^{\frac{1}{1-q}}.
\label{est4-thm3}
\end{equation}
\end{theorem}

In the next theorem we provide criteria for the existence of positive
solutions for the equation 
\begin{equation}
-\mathcal{L}u+u^{q}V=f\quad \mathrm{in}\,\,\Omega  \label{diff}
\end{equation}%
under some additional assumptions and give two-sided pointwise estimates for
these solutions.

\begin{theorem}
\label{C4}Let $M$ be a weighted manifold and $\Omega \subset M$ be a
connected relatively compact open set with smooth boundary. Let $f\geq 0$
and $V$ be locally H\"{o}lder continuous functions in $\Omega $ and in
addition $f\in C\left( \overline{\Omega }\right) $. Set $h=G^{\Omega }f$.
Then the following statements hold.

$\left( i\right) $ For $q>1$ and $V\leq 0$, suppose that for all $x\in
\Omega $%
\begin{equation}
-G^{\Omega }\left( h^{q}V\right) \left( x\right) \leq \left( 1-\frac{1}{q}%
\right) ^{q}\frac{1}{q-1}\,h(x).  \label{cond1-corr}
\end{equation}%
Then \emph{(\ref{diff})} has a nonnegative solution $u\in C^{2}\left( \Omega
\right) \cap C\left( \overline{\Omega }\right) $, and it satisfies for all $%
x\in \Omega $%
\begin{equation}
\frac{h(x)}{\left[ 1+(q-1)\frac{G^{\Omega }(h^{q}V)(x)}{h(x)}\right] ^{\frac{%
1}{q-1}}}\leq u\left( x\right) \leq \frac{q}{q-1}\,h(x).  \label{uh+}
\end{equation}

$\left( ii\right) $ For $q<0$ and $V\geq 0$, suppose that for all $x\in
\Omega $%
\begin{equation}
G^{\Omega }(h^{q}V)(x)\leq \left( 1-\frac{1}{q}\right) ^{q}\frac{1}{1-q}%
\,h(x).  \label{cond2-corr}
\end{equation}%
Then \emph{(\ref{diff})} has a nonnegative solution $u\in C^{2}\left( \Omega
\right) \cap C\left( \overline{\Omega }\right) $, and it satisfies for all $%
x\in \Omega $%
\begin{equation}
\frac{1}{1-\frac{1}{q}}\,h(x)\leq u(x)\leq \left[ 1-(1-q)\frac{G^{\Omega
}(h^{q}V)(x)}{h(x)}\right] ^{\frac{1}{1-q}}h(x).  \label{uh-}
\end{equation}
\end{theorem}

Note that the terms in square brackets in both (\ref{uh+}) and (\ref{uh-})
are positive and $<1$; it follows that in both cases $\left( i\right) $ and $%
\left( ii\right) $ $u\simeq h$ in $\Omega $. Since $h\left( x\right) \simeq
d_{\Omega }\left( x\right) :=\limfunc{dist}\left( x,\partial \Omega \right) $%
, we obtain $u(x)\simeq d_{\Omega }(x)$.

In the next theorem we give an abstract version of Theorem \ref{C4} that
provides an existence result together with pointwise estimates of solutions $%
u$ for the following integral equation with $q\in \mathbb{R}\setminus
\left\{ 0\right\} $:%
\begin{equation}
u(x)+\int_{\Omega }K(x,y)\,u(y)^{q}\,V(y)\,dm(y)=h(x)\quad dm-\mathrm{%
a.e.\,\,in}\,\,\Omega .  \label{int}
\end{equation}%
Here $(\Omega, m)$ is a measure space with $\sigma $-finite nonnegative
measure $m$, $0<u<\infty $ $dm$-a.e., and $K:\Omega \times \Omega
\rightarrow {\bar{\mathbb{R}}_{+}}\cup \{+\infty \}$ is a nonnegative
measurable kernel.

The coefficient $V$ is assumed to be a measurable function in $\Omega $ with
a definite sign (either $V\geq 0$, or $V\leq 0$). In fact, we can use $%
d\omega $ in place of $V\,dm$, with an arbitrary $\sigma $-finite measure $%
\omega $ (either nonnegative, or nonpositive) in $\Omega $, where $%
0<u<+\infty $ $d\omega $-a.e., and the integral equation holds $d\omega $%
-a.e.

For a nonnegative Borel measure $\mu $ in $\Omega $, we will write 
\begin{equation*}
K\mu (x)=\int_{\Omega }K(x,y)\,d\mu (y),
\end{equation*}%
and $Kf(x)=K(fdm)(x)$ for a nonnegative measurable function $f$.

\begin{theorem}
\label{thm3}\label{T4} Let $(\Omega,m)$ be a measure space with $\sigma $%
-finite measure $m$, and let $K$ be a nonnegative kernel on $\Omega \times
\Omega $. Let $h$ be a measurable function such that 
\begin{equation}
0<h<+\infty \quad dm\mathrm{-a.e.}\,\,\mathrm{in}\,\,\Omega .  \label{cond-h}
\end{equation}%
Let $V$ be a measurable function in $\Omega $. Then the following statements
hold.

$\left( i\right) $ For $q>1$, and $V\leq 0$, suppose that the following
condition holds, 
\begin{equation}
-K(h^{q}V)(x)\leq \left( 1-\frac{1}{q}\right) ^{q}\frac{1}{q-1}\,h(x)\quad dm%
\mathrm{-a.e.}\,\,\mathrm{in}\,\,\Omega .  \label{condV-}
\end{equation}%
Then \emph{(\ref{int})} has a minimal positive solution $u$, and it satisfies%
\begin{equation}
h(x)\leq u(x)\leq \frac{q}{q-1}\,h(x)\quad \mathrm{in}\,\,\Omega .
\label{ineqV-}
\end{equation}

$\left( ii\right) $ For $q<0$ and $V\geq 0$, suppose that the following
condition holds, 
\begin{equation}
K(h^{q}V)(x)\leq \left( 1-\frac{1}{q}\right) ^{q}\frac{1}{1-q}\,h(x)\quad dm-%
\mathrm{a.e.}\,\,\mathrm{in}\,\,\Omega .  \label{condV+}
\end{equation}%
Then \emph{(\ref{int})} has a maximal positive solution $u$, and it satisfies%
\begin{equation}
\frac{1}{1-\frac{1}{q}}\,h(x)\leq u(x)\leq h(x)\quad dm-\mathrm{a.e.}%
\,\,\quad \mathrm{in}\,\,\Omega .  \label{ineqV+}
\end{equation}
\end{theorem}

\begin{remark}
\RM Statement $(i)$ of Theorem \ref{T4} is essentially known, and we include
it here only for the sake of completeness. It holds under a less restrictive
assumption 
\begin{equation}
-K(H^{q}V)(x)\leq \left( 1-\frac{1}{q}\right)^{q^2} \frac{1}{(q-1)^q}
\,H(x)\quad dm\mathrm{-a.e.}\,\,\mathrm{in}\,\,\Omega ,  \label{condV2-}
\end{equation}%
where $H= -K(h^{q}V)$; in this case, $u\simeq h+H$ (see \cite{KV}).
\end{remark}

\section{Some auxiliary material}

\label{SecAux}In this section we prove some lemmas needed for the proofs of
Theorems \ref{T1}, \ref{T2}. Everywhere $M$ stands for an arbitrary weighted
manifold.

\begin{lemma}
\label{L1}Let $v,h$ be $C^{2}$-functions in $\Omega \subseteq M$, and $\phi $
be a $C^{2}$-function on an interval $I\subset \mathbb{R}$ such that $%
v\left( \Omega \right) \subset I$. Then the following identity is true:%
\begin{equation}
\mathcal{L}\left( h\phi \left( v\right) \right) =\phi ^{\prime }(v)\mathcal{L%
}\left( hv\right) +\phi ^{\prime \prime }(v)|\nabla v|^{2}h+\left( \phi
(v)-v\phi ^{\prime }(v)\right) \mathcal{L}h.  \label{Lh}
\end{equation}%
Consequently, if $\phi ^{\prime }\neq 0$ then%
\begin{equation}
-\mathcal{L}\left( hv\right) =-\frac{\mathcal{L}\left( h\phi (v)\right) }{%
\phi ^{\prime }\left( v\right) }+\frac{\phi ^{\prime \prime }(v)}{\phi
^{\prime }\left( v\right) }|\nabla v|^{2}h+\left( \frac{\phi (v)}{\phi
^{\prime }\left( v\right) }-v\right) \mathcal{L}h.  \label{Lh-id}
\end{equation}
\end{lemma}

\begin{proof}
For functions $u\in C^{2}\left( \Omega \right) $, consider the following
operator%
\begin{equation*}
\widetilde{\mathcal{L}}u=\frac{1}{h^{2}}\func{div}_{\omega }(h^{2}\nabla u)=%
\frac{1}{\omega h^{2}}\func{div}(\omega h^{2}\nabla u) ,
\end{equation*}%
that is, the weighted Laplace operator of the weighted manifold $\left(
\Omega,h^{2}dm\right) =\left( \Omega ,\omega h^{2}dm_{0}\right) .$ Using the
product rule for $\func{div}_{\omega }$, we obtain%
\begin{equation*}
\widetilde{\mathcal{L}}u=\mathcal{L}u+2\langle \frac{\nabla h}{h},\nabla
u\rangle .
\end{equation*}%
On the other hand, by the product rule (\ref{product}) for $\mathcal{L}$ we
have%
\begin{equation*}
\mathcal{L}\left( hu\right) =h\mathcal{L}u+2\langle \nabla h,\nabla u\rangle
+u\mathcal{L}h,
\end{equation*}%
which implies the identity%
\begin{equation}
\mathcal{L}\left( hu\right) =h\widetilde{\mathcal{L}}u+u\mathcal{L}h.
\label{hv}
\end{equation}%
Using (\ref{hv}) with $u=\phi \left( v\right) $ and applying the chain rule (%
\ref{chain}) for $\widetilde{\mathcal{L}}$, we obtain%
\begin{align*}
\mathcal{L}\left( h\phi (v)\right) & =h\widetilde{\mathcal{L}}\phi (v)+\phi
(v)\mathcal{L}h \\
& =h\left( \phi ^{\prime }(v)\widetilde{\mathcal{L}}v+\phi ^{\prime \prime
}(v)|\nabla v|^{2}\right) +\phi (v)\mathcal{L}h \\
& =\phi ^{\prime }(v)(h\widetilde{\mathcal{L}}v+v\mathcal{L}h)+\phi ^{\prime
\prime }(v)|\nabla v|^{2}h+\left( \phi (v)-v\phi ^{\prime }(v)\right) 
\mathcal{L}h \\
& =\phi ^{\prime }(v)\mathcal{L}\left( hv\right) +\phi ^{\prime \prime
}(v)|\nabla v|^{2}h+\left( \phi (v)-v\phi ^{\prime }(v)\right) \mathcal{L}h,
\end{align*}%
which proves (\ref{Lh}). Then (\ref{Lh-id}) follows immediately from (\ref%
{Lh}).
\end{proof}

\begin{lemma}
\label{ident} Let $\phi $ be a $C^{2}$ function on an interval $I\subset 
\mathbb{R}$ such that $\phi >0$ and $\phi ^{\prime }>0$ in $I$. For two
functions $v,h\in C^{2}\left( \Omega \right) $, $h>0$, set 
\begin{equation*}
u=h\phi \left( v\right)
\end{equation*}%
assuming that $\phi \left( v\right) $ is well-defined, that is, $v\left(
\Omega \right) \subset I.$

If the function $u$ satisfies the inequality%
\begin{equation}
-\mathcal{L}u+Vu^{q}\geq -\mathcal{L}h  \label{uV}
\end{equation}%
in $\Omega $, where $V\in C\left( \Omega \right) $, $q\in \mathbb{R}%
\setminus \left\{ 0\right\} $, then the function $v$ satisfies in $\Omega $
the inequality%
\begin{equation}
-\mathcal{L}\left( hv\right) +h^{q}V\frac{\phi (v)^{q}}{\phi ^{\prime }(v)}%
\geq \left( \frac{\phi (v)-1}{\phi ^{\prime }(v)}-v\right) \mathcal{L}h+%
\frac{\phi ^{\prime \prime }(v)}{\phi ^{\prime }(v)}|\nabla v|^{2}h.
\label{ident-Lh}
\end{equation}

If instead $u$ satisfies%
\begin{equation}
-\mathcal{L}u+Vu^{q}\leq -\mathcal{L}h,  \label{uV-}
\end{equation}%
then \emph{(\ref{ident-Lh})} holds with $\leq $ instead of $\geq $.
\end{lemma}

\begin{proof}
It follows from $u=h\phi (v)$ and (\ref{uV}) that%
\begin{equation}
\mathcal{L}\left( h\phi \left( v\right) \right) \leq h^{q}V\phi \left(
u\right) ^{q}+\mathcal{L}h.  \label{semilinear2a}
\end{equation}%
Substituting this into (\ref{Lh-id}) we obtain%
\begin{equation*}
-\mathcal{L}\left( hv\right) \geq -\frac{h^{q}V\phi \left( u\right) ^{q}+%
\mathcal{L}h}{\phi ^{\prime }\left( v\right) }+\frac{\phi ^{\prime \prime
}(v)}{\phi ^{\prime }\left( v\right) }|\nabla v|^{2}h+\left( \frac{\phi (v)}{%
\phi ^{\prime }\left( v\right) }-v\right) \mathcal{L}h,
\end{equation*}%
whence (\ref{ident-Lh}) follows. The second claim is proved in the same way.
\end{proof}

\begin{lemma}
\label{Cor} Under the hypotheses of Lemma \emph{\ref{ident}}, assume in
addition that $\mathcal{L}h\leq 0$ in $\Omega $ and $0\in I$. If in $I$ 
\begin{equation}
\phi (0)=1,\quad \phi ^{\prime }>0,\quad \phi ^{\prime \prime }\geq 0,
\label{conv-incr}
\end{equation}%
then the function $v$ satisfies the following differential inequality in $%
\Omega $%
\begin{equation}
-\mathcal{L}\left( hv\right) +h^{q}V\frac{\phi (v)^{q}}{\phi ^{\prime }(v)}%
\geq 0.  \label{Lh+}
\end{equation}%
If instead of \emph{(\ref{conv-incr})} we have 
\begin{equation}
\phi (0)=1,\quad \phi ^{\prime }>0,\quad \phi ^{\prime \prime }\leq 0,
\label{conv-incr-a}
\end{equation}%
then $v$ satisfies in $\Omega $ 
\begin{equation}
-\mathcal{L}\left( hv\right) +h^{q}V\frac{\phi (v)^{q}}{\phi ^{\prime }(v)}%
\leq 0.  \label{Lh-}
\end{equation}
\end{lemma}

\begin{proof}
Consider the case (\ref{conv-incr}). By the mean value theorem, for any $%
v\in I$ there exists $\xi \in \left[ 0,v\right] $ such that%
\begin{equation*}
\frac{\phi \left( v\right) -1}{v}=\frac{\phi \left( v\right) -\phi \left(
0\right) }{v}=\phi ^{\prime }\left( \xi \right) .
\end{equation*}%
By the convexity of $\phi $ we obtain $\phi ^{\prime }\left( \xi \right)
\leq \phi ^{\prime }\left( v\right) $ provided $v>0$, that is%
\begin{equation*}
\frac{\phi \left( v\right) -1}{v}\leq \phi ^{\prime }\left( v\right) \ \ 
\text{for }v>0,
\end{equation*}%
and the opposite inequality in the case $v<0$. It follows that, for all $%
v\in I$,%
\begin{equation*}
\frac{\phi (v)-1}{\phi ^{\prime }(v)}-v\leq 0.
\end{equation*}%
Substituting into (\ref{ident-Lh}) and using also $\mathcal{L}h\leq 0$ and (%
\ref{conv-incr}), we obtain (\ref{Lh+}). The proof in the case (\ref%
{conv-incr-a}) is similar.
\end{proof}

\begin{remark}
\RM Note that in the case $\mathcal{L}h\equiv 0$ the condition $\phi \left(
0\right) =1$ in (\ref{conv-incr}) and (\ref{conv-incr-a}) is not required as
in this case the term 
\begin{equation*}
\left( \frac{\phi (v)-1}{\phi ^{\prime }(v)}-v\right) \mathcal{L}h
\end{equation*}
vanishes identically.
\end{remark}

\begin{lemma}
\label{max}Suppose $\Omega $ is an open subset of $M$ and $F$ is a l.s.c. $%
\mathcal{L}$-superharmonic function in $\Omega $. Suppose $F=F_{1}+F_{2}$,
where 
\begin{equation}
\liminf_{x\rightarrow \partial _{\infty }\Omega }F_{1}(x)\geq 0\ \ \text{%
and\ \ }F_{2}\geq -P,  \label{F12}
\end{equation}%
where $P=G^{\Omega }\mu $ is a Green potential of a positive measure $\mu $
in $\Omega $ so that $P\not\equiv +\infty $ on every component of $\Omega $.
Then $F\geq 0$ in $\Omega $.
\end{lemma}

\begin{proof}
Indeed, the function $F+P$ is obviously superharmonic, and $F+P\geq F_{1}$.
Hence $\liminf_{x\rightarrow \partial _{\infty }\Omega }(F+P)(x)\geq 0$, and
by the standard form of the maximum principle $F+P\geq 0$ on $\Omega $ (cf. 
\cite{AG}, \cite[Sec. 5.4.3]{G1}). Hence $F$ is a superharmonic majorant of $%
-P$, whose least superharmonic majorant must be zero (with the same proof as
in the classical case \cite[Theorem 4.2.6]{AG}), which yields $F\geq 0$.
\end{proof}

The following version of the maximum principle will be frequently used.

\begin{lemma}
\label{Lemmax}Let $\Omega $ be an open subset of $M$ and let $v\in
C^{2}\left( \Omega \right) $ satisfy%
\begin{equation*}
\left\{ 
\begin{array}{ll}
-\mathcal{L}v\geq f & \text{in\ }\Omega , \\ 
\liminf_{x\rightarrow \partial _{\infty }\Omega }v\left( x\right) \geq 0, & 
\end{array}%
\right.
\end{equation*}%
where $f\in C\left( \Omega \right) $ such that $G^{\Omega }f$ is well
defined in $\Omega $. Then for all $x\in \Omega $%
\begin{equation}
v\left( x\right) \geq G^{\Omega }f\left( x\right) .  \label{vx}
\end{equation}
\end{lemma}

\begin{proof}
If $G^{\Omega }f_{-}=+\infty $ then (\ref{vx}) is trivially satisfied.
Hence, assume in the sequel that $G^{\Omega }f_{-}<\infty .$ Let us
approximate $f$ from below by a sequence $\left\{ f_{n}\right\} $ of $C^{1}$
functions in $\Omega $ such that $f_{n}\uparrow f$ as $n\rightarrow \infty \ 
$and $G^{\Omega }f_{n}^{-}<\infty $ (where $f_{n}^{\pm }:=\left(
f_{n}\right) _{\pm }$). Moreover, we can also assume that $f_{n}^{+}$ is
compactly supported in $\Omega $.

Fix $n$ and consider in $\Omega $ two functions%
\begin{equation*}
F_{1}=v+G^{\Omega }f_{n}^{-}\ \ \ \text{and\ \ \ }F_{2}=-G^{\Omega
}f_{n}^{+}.
\end{equation*}%
The hypotheses (\ref{F12}) of Lemma \ref{max} are obviously satisfied. The
function%
\begin{equation*}
F=v+G^{\Omega }f_{n}^{-}-G^{\Omega }f_{n}^{+}
\end{equation*}%
is superharmonic in $\Omega $ since%
\begin{equation*}
-\mathcal{L}F=-\mathcal{L}v+f_{n}^{-}-f_{n}^{+}=f-f_{n}\geq 0.
\end{equation*}%
By Lemma \ref{max} we conclude that $F\geq 0$ in $\Omega $ and, hence, 
\begin{equation*}
v\geq G^{\Omega }f_{n}^{+}-G^{\Omega }f_{n}^{-}.
\end{equation*}%
Letting $n\rightarrow \infty $ and using the convergence theorems we obtain (%
\ref{vx})
\end{proof}

\section{Proof of Theorem \protect\ref{T2}}

\label{SecProofB}We start the proof with a particular case of Theorem \ref%
{T2} where the idea of the proof is most transparent and not buried in
technical complications.

\begin{proof}[Proof of Theorem \protect\ref{T2} in the special case $h>0$, $%
u>0$ in $\overline{\Omega }$, and $V\in C(\overline{\Omega })$]
In this case the function $G^{\Omega }(h^{q}V)\left( x\right) $ is finite
for all $x\in \Omega $.

Choose a function $\phi $ (to be used in Lemma~\ref{Cor}) to solve the
initial value problem 
\begin{equation}
\phi ^{\prime }(s)=\phi (s)^{q},\ \ \phi \left( 0\right) =1.  \label{fi'}
\end{equation}%
For $q=1$ this gives 
\begin{equation}
\phi (s)=e^{s},\quad s\in \mathbb{R},  \label{exp}
\end{equation}%
while for $q\not=1$ we obtain 
\begin{equation}
\phi (s)=[(1-q)s+1]^{\frac{1}{1-q}},\quad s\in I_{q},  \label{phi}
\end{equation}%
where the domain $I_{q}$ of $\phi $ is given by: 
\begin{equation}
I_{q}=%
\begin{cases}
\,(-\infty ,\frac{1}{q-1})\quad & \text{if}\,\,q>1, \\ 
\,(-\infty ,+\infty )\quad & \text{if}\,\,q=1, \\ 
\,(-\frac{1}{1-q},+\infty ) & \text{if}\,\,q<1%
\end{cases}
\label{Iq}
\end{equation}%
(see Fig. \ref{fig}).

\begin{figure}[tbh]
\centering
\includegraphics[scale=1.]{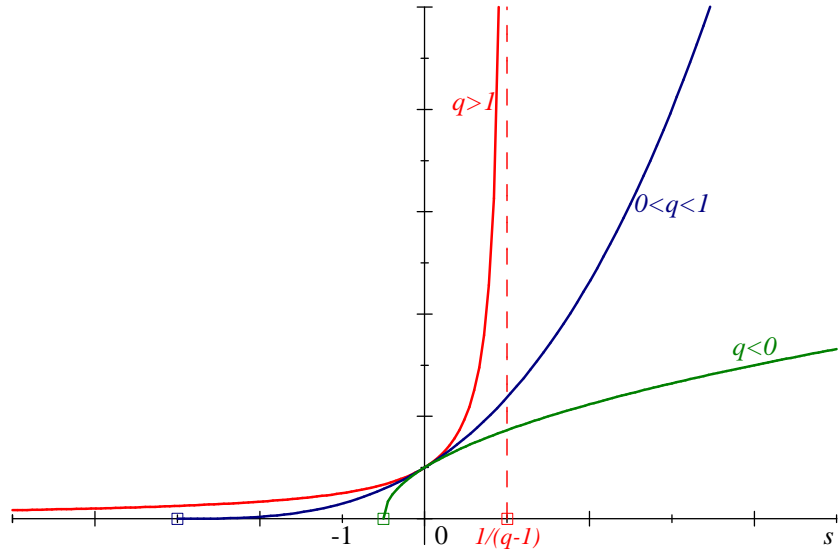}
\caption{Examples of the function $\protect\phi $ in three cases $q>1$, $%
0<q<1$, $q<0$. The boxed points have the abscissa $\frac{1}{q-1}$.}
\label{fig}
\end{figure}

\bigskip Note that in all cases $\phi \left( I_{q}\right) =\left( 0,\infty
\right) $. Also we have 
\begin{equation}
\phi ^{\prime }(s)=[(1-q)s+1]^{\frac{q}{1-q}},\quad \phi ^{\prime \prime
}(s)=q[(1-q)s+1]^{\frac{2q-1}{1-q}}.  \label{der}
\end{equation}%
In particular, $\phi ^{\prime }>0$ in $I_{q}$, whereas $\phi ^{\prime \prime
}>0$ for $q>0$ and $\phi ^{\prime \prime }<0$ for $q<0$. Consequently, the
inverse function $\phi ^{-1}$ is well-defined on $\left( 0,\infty \right) $.

In the case $0<q<1$ it will be convenient for us to extend the domain of $%
\phi $ to all $s\leq -\frac{1}{1-q}$ by setting $\phi \left( s\right) =0$ so
that in this case we have for all $s\in \left( -\infty ,\infty \right) $%
\begin{equation}
\phi \left( s\right) =[(1-q)s+1]_{+}^{\frac{1}{1-q}}.  \label{fi+}
\end{equation}

Observe that all the estimates (\ref{main-exp}), (\ref{main1}), (\ref{main})
that we need to prove in the case $q>0$ can be written in the unified form%
\begin{equation}
\frac{u(x)}{h(x)}\geq \phi \left( -\frac{1}{h(x)}G^{\Omega
}(h^{q}V)(x)\right) ,  \label{u/h}
\end{equation}%
for all $x\in \Omega $. Similarly, estimate (\ref{main-negative}) in the
case $q<0$ is equivalent to the opposite inequality%
\begin{equation}
\frac{u(x)}{h(x)}\leq \phi \left( -\frac{1}{h(x)}G^{\Omega
}(h^{q}V)(x)\right) .  \label{u/h-}
\end{equation}
Since by hypothesis the functions $h$ and $u$ are positive in $\overline{%
\Omega }$, the function 
\begin{equation}
v=\phi ^{-1}\left( \frac{u}{h}\right)  \label{v}
\end{equation}%
is well-defined in $\overline{\Omega }$ and belongs to the class $%
C^{2}\left( \Omega \right) \cap C\left( \overline{\Omega }\right) $.

Consider first the case $q>0$. In this case we will deduce (\ref{u/h}) from
the following inequality for $v$:%
\begin{equation}
v\left( x\right) \geq -\frac{1}{h\left( x\right) }G^{\Omega }(h^{q}V)\left(
x\right) ,  \label{lower-v}
\end{equation}%
for all $x\in \Omega $. Indeed, if (\ref{lower-v}) holds then applying $\phi 
$ to both sides of (\ref{lower-v}) and observing that $\phi \left( v\right) =%
\frac{u}{h}$, we obtain (\ref{u/h}). However, we should first verify that
the both sides of (\ref{lower-v}) are in the domain of $\phi $. In the cases 
$q=1$ and $0<q<1$ the (extended) domain of $\phi $ is $\left( -\infty
,+\infty \right) $, so that there is no problem. In the case $q>1$ we have $%
v\left( x\right) \in I_{q}=\,(-\infty ,\frac{1}{q-1})$ by (\ref{v}), which
implies that the right hand side of (\ref{lower-v}), being bounded by $%
v\left( x\right) $, is also in $I_{q}$. This argument also shows that in $%
\Omega $ 
\begin{equation*}
\frac{1}{q-1}>-\frac{1}{h\left( x\right) }G^{\Omega }(h^{q}V)\left( x\right)
,
\end{equation*}%
which proves (\ref{main-cond}).

To prove (\ref{lower-v}) observe that the function $u=h\phi \left( v\right) $
satisfies 
\begin{equation*}
-\mathcal{L}u+Vu^{q}\geq -\mathcal{L}h\geq 0
\end{equation*}%
in $\Omega $ as required by Lemma \ref{Cor}. In the case $q>0$ the function $%
\phi $ satisfies (\ref{conv-incr}), and we obtain by inequality (\ref{Lh+})
of Lemma \ref{Cor} and by (\ref{fi'}) that in $\Omega $ 
\begin{equation}
-\mathcal{L}\left( hv\right) +h^{q}V\geq 0.  \label{ineq}
\end{equation}%
Since $u\geq h$ on $\partial \Omega $, it follows that on $\partial \Omega $%
\begin{equation*}
hv=h\phi ^{-1}\left( \frac{u}{h}\right) \geq h\phi ^{-1}\left( 1\right) =0.
\end{equation*}%
Since $hv$ satisfies (\ref{ineq}) and the boundary condition $hv\geq 0$ on $%
\partial \Omega $, we obtain by the maximum principle that in $\Omega $%
\begin{equation}
hv\geq -G^{\Omega }\left( h^{q}V\right) ,  \label{hv>}
\end{equation}%
which is equivalent to (\ref{lower-v}).

Consider now the case $q<0$. Then we have 
\begin{equation*}
-\mathcal{L}u+Vu^{q}\leq -\mathcal{L}h
\end{equation*}%
and, hence, obtain by inequality (\ref{Lh-}) of Lemma~\ref{Cor} and (\ref%
{fi'}) that in $\Omega $, 
\begin{equation}
-\mathcal{L}\left( hv\right) +h^{q}V\leq 0.  \label{ineq-a}
\end{equation}%
In this case we have $u\leq h$ on $\partial \Omega $, which implies $hv\leq
0 $ on $\partial \Omega $. Using (\ref{ineq-a}) with this boundary
condition, we obtain that in $\Omega $ 
\begin{equation*}
hv\leq -G^{\Omega }(h^{q}V)
\end{equation*}%
and, hence,%
\begin{equation}
v\leq -\frac{1}{h}G^{\Omega }(h^{q}V).  \label{upper-v1}
\end{equation}%
Since $v\left( x\right) \in I_{q}=(-\frac{1}{1-q},+\infty )$, it follows
that both sides of (\ref{upper-v1}) belong to $I_{q}$. Consequently, we have%
\begin{equation*}
-\frac{1}{1-q}<-\frac{1}{h}G^{\Omega }(h^{q}V),
\end{equation*}%
which proves (\ref{main-cond}). Applying $\phi $ to both sides of (\ref%
{upper-v1}), we obtain (\ref{u/h-}) and, hence, (\ref{main-negative}).
\end{proof}

\begin{proof}[Proof of Theorem~\protect\ref{T2} in the general case]
We will use the same function $\phi $ as defined above by (\ref{exp})-(\ref%
{phi}), but it will be convenient to extend the domain $I_{q}$ of $\phi $ to
the endpoints of the interval $I_{q}$ by taking the limits of $\phi $ at the
endpoints. The extended domain of $\phi $ is therefore the interval%
\begin{equation*}
\overline{I}_{q}:=%
\begin{cases}
\,[-\infty ,\frac{1}{q-1}]\quad & \text{if}\,\,q>1, \\ 
\,[-\infty ,+\infty ]\quad & \text{if}\,\,q=1, \\ 
\,[-\frac{1}{1-q},+\infty ] & \text{if}\,\,q<1.%
\end{cases}%
\end{equation*}%
Moreover, in the case $0<q<1$ we extend $\phi \left( s\right) $ further to
all $s\in \left[ -\infty ,+\infty \right] $ by using (\ref{fi+}).

With these extensions the required estimates (\ref{main-exp}), (\ref{main1})
and (\ref{main}) in the case $q>0$ can be written in the unified form (\ref%
{u/h}), and the estimate (\ref{main-negative}) -- in the form (\ref{u/h-}).

Consider first the case $q>0$. For any $\varepsilon >0$, set%
\begin{equation*}
u_{\varepsilon }=u+\varepsilon
\end{equation*}%
and define the function $v_{\varepsilon }$ in $\Omega $ via 
\begin{equation*}
v_{\varepsilon }=\phi ^{-1}\left( \frac{u_{\varepsilon }}{h}\right) ,
\end{equation*}%
where $\phi $ is the same as above. Since $u_{\varepsilon }$ and $h$ are
positive in $\Omega $, the function $v_{\varepsilon }$ is well-defined in $%
\Omega $ and belongs to $C^{2}(\Omega )$. Note also that $v_{\varepsilon
}\left( \Omega \right) \subset I_{q}$.

Applying identity (\ref{Lh-id}) to functions $h,v_{\varepsilon }\in
C^{2}\left( \Omega \right) $, we obtain%
\begin{equation*}
-\mathcal{L}\left( hv_{\varepsilon }\right) =-\frac{\mathcal{L}\left( h\phi
(v_{\varepsilon })\right) }{\phi ^{\prime }\left( v\right) }+\frac{\phi
^{\prime \prime }(v_{\varepsilon })}{\phi ^{\prime }\left( v_{\varepsilon
}\right) }|\nabla v_{\varepsilon }|^{2}h+\left( \frac{\phi (v_{\varepsilon })%
}{\phi ^{\prime }\left( v_{\varepsilon }\right) }-v_{\varepsilon }\right) 
\mathcal{L}h.
\end{equation*}%
Since%
\begin{equation*}
-\mathcal{L}\left( h\phi (v_{\varepsilon })\right) =-\mathcal{L}%
u_{\varepsilon }=-\mathcal{L}u,
\end{equation*}%
it follows%
\begin{equation}
-\mathcal{L}(hv_{\varepsilon })=\frac{-\mathcal{L}u}{\phi ^{\prime
}(v_{\varepsilon })}+\frac{\phi ^{\prime \prime }(v_{\varepsilon })}{\phi
^{\prime }\left( v_{\varepsilon }\right) }|\nabla v_{\varepsilon
}|^{2}h+\left( \frac{\phi (v_{\varepsilon })}{\phi ^{\prime }\left(
v_{\varepsilon }\right) }-v_{\varepsilon }\right) \mathcal{L}h.  \label{he}
\end{equation}%
Observe also that by (\ref{fi'})%
\begin{equation}
\phi ^{\prime }(v_{\varepsilon })=\phi (v_{\varepsilon })^{q}=\left( \frac{%
u_{\varepsilon }}{h}\right) ^{q}.  \label{fe}
\end{equation}%
Since $q>0$, we have by (\ref{semilinearA}) 
\begin{equation*}
-\mathcal{L}u\geq -Vu^{q}-\mathcal{L}h.
\end{equation*}%
Substituting this and (\ref{fe}) into (\ref{he}), we obtain 
\begin{equation*}
-\mathcal{L}(hv_{\varepsilon })\geq -h^{q}\left( \frac{u}{u_{\varepsilon }}%
\right) ^{q}V+\frac{\phi ^{\prime \prime }(v_{\varepsilon })}{\phi ^{\prime
}\left( v_{\varepsilon }\right) }|\nabla v_{\varepsilon }|^{2}h+\left( \frac{%
\phi (v_{\varepsilon })-1}{\phi ^{\prime }\left( v_{\varepsilon }\right) }%
-v_{\varepsilon }\right) \mathcal{L}h.
\end{equation*}%
Since $\phi $ satisfies (\ref{conv-incr}) and, hence, the last two terms on
the right-hand side of the preceding inequality are nonnegative (cf. the
proof of Lemma \ref{Cor}), we arrive at 
\begin{equation}
-\mathcal{L}(hv_{\varepsilon })\geq -h^{q}\left( \frac{u}{u_{\varepsilon }}%
\right) ^{q}V\ \ \text{in }\Omega .  \label{Lhv}
\end{equation}%
In the case $q\not=1$, $q>0$ we have by (\ref{phi})%
\begin{equation*}
\phi ^{-1}\left( s\right) =\frac{s^{1-q}-1}{1-q},\ \ s>0,
\end{equation*}%
and, hence, in $\Omega $ 
\begin{equation*}
hv_{\varepsilon }=h\phi ^{-1}\left( \frac{u_{\varepsilon }}{h}\right) =\frac{%
1}{1-q}\left( h^{q}u_{\varepsilon }^{1-q}-h\right) .
\end{equation*}%
It follows that, for all $y\in \partial \Omega $, 
\begin{equation*}
\lim_{x\rightarrow y,\,x\in \Omega }h(x)v_{\varepsilon }(x)=\frac{1}{1-q}%
\left( h^{q}(y)u_{\varepsilon }(y)^{1-q}-h(y)\right) \geq 0,
\end{equation*}%
since $u_{\varepsilon }(y)\geq h(y)+\varepsilon >h(y)$.

For $q=1$ we have $\phi ^{-1}\left( s\right) =\ln s$ and, hence, in $\Omega $%
\begin{equation}
hv_{\varepsilon }=h\ln \left( \frac{u_{\varepsilon }}{h}\right) .
\label{hve}
\end{equation}%
For any $y\in \partial \Omega $ such that $h(y)>0$, we obtain%
\begin{equation*}
\lim_{x\rightarrow y,\,x\in \Omega }h(x)v_{\varepsilon }(x)=h(y)\ln \left( 
\frac{u_{\varepsilon }(y)}{h(y)}\right) >0,
\end{equation*}%
and if $h(y)=0$, then, using $u_{\varepsilon }\geq \varepsilon $, we obtain
from (\ref{hve})%
\begin{equation}
\lim_{x\rightarrow y,\,x\in \Omega }h(x)v_{\varepsilon }(x)=0.  \label{hvd}
\end{equation}%
Hence, in the case $q>0$, we can extend $hv_{\varepsilon }$ by continuity to 
$\overline{\Omega }$ so that $hv_{\varepsilon }\in C(\overline{\Omega })\cap
C^{2}(\Omega )$ and 
\begin{equation*}
hv_{\varepsilon }\geq 0\ \ \text{on}\,\,\partial \Omega .
\end{equation*}%
Note that $h^{q}\left( \frac{u}{u_{\varepsilon }}\right) ^{q}V\in C(\Omega ) 
$ and $G^{\Omega }\left( h^{q}\left( \frac{u}{u_{\varepsilon }}\right)
^{q}V\right) $ is well-defined in $\Omega $, since 
\begin{equation*}
G^{\Omega }\left( h^{q}\left( \frac{u}{u_{\varepsilon }}\right) ^{q}V_{\pm
}\right) \leq G^{\Omega }\left( h^{q}V_{\pm }\right) ,
\end{equation*}%
and $G^{\Omega }\left( h^{q}V\right) $ is well-defined by hypothesis. Hence,
by the maximum principle of Lemma \ref{Lemmax}, we conclude from (\ref{Lhv})
and (\ref{hvd}) that%
\begin{equation*}
hv_{\varepsilon }\geq -G^{\Omega }\left( h^{q}\left( \frac{u}{u_{\varepsilon
}}\right) ^{q}V\right)
\end{equation*}%
and, hence, 
\begin{equation}
v_{\varepsilon }\geq -\frac{1}{h}G^{\Omega }\left( h^{q}\left( \frac{u}{%
u_{\varepsilon }}\right) ^{q}V\right) \text{ in }\Omega .  \label{eps-v}
\end{equation}

Assume now $q\geq 1$. Assume also that $G^{\Omega }\left( h^{q}V_{+}\right)
\not\equiv +\infty $ in $\Omega $, because otherwise, (\ref{main-exp}), (\ref%
{main-cond}) and (\ref{main1}) are trivially satisfied, and so there is
nothing to prove. Let us first show that under these assumptions $u>0$ in $%
\Omega $. Observe that if $G^{\Omega }\left( h^{q}V_{+}\right) \not\equiv
+\infty $ in $\Omega $, then $G^{\Omega }\left( h^{q}V_{+}\right)
(x)<+\infty $ for every $x\in \Omega $. Indeed, for an open set $\Omega
^{\prime }\Subset \Omega $ with smooth boundary, fix a function $\eta \in
C_{0}^{\infty }(\Omega )$ such that $\eta =1$ in $\Omega ^{\prime }$. Then
the function $G^{\Omega }\left( h^{q}V_{+}\right) -G^{\Omega }\left( \eta
h^{q}V_{+}\right) $ is harmonic in $\Omega ^{\prime }$, and $G^{\Omega
}\left( \eta h^{q}V_{+}\right) $ is bounded in $\Omega $ since $\eta
h^{q}V_{+}\in C(\overline{\Omega })$. Consequently, $G^{\Omega }\left(
h^{q}V_{+}\right) $ is finite in $\Omega ^{\prime }$, and hence in $\Omega $.

It follows from (\ref{eps-v}) and $u\leq u_{\varepsilon }$, that%
\begin{equation}
v_{\varepsilon }\geq -\frac{1}{h}G^{\Omega }\left( h^{q}V_{+}\right) .
\label{vI}
\end{equation}%
Since the value $v_{\varepsilon }=\phi ^{-1}\left(\frac{u_{\varepsilon }}{h}%
\right) $ belongs to $I_{q}$ and the value of the right hand side of (\ref%
{vI}) lies in $[-\infty ,0]$, which, in the present case $q\geq 1$, is
contained in $\overline{I}_{q}$, we can apply $\phi $ to both sides of this
inequality and obtain%
\begin{equation}
u_{\varepsilon }\geq h\phi \left( -\frac{G^{\Omega }\left( h^{q}V_{+}\right) 
}{h}\right) .  \label{est-l}
\end{equation}%
Letting $\varepsilon \rightarrow 0$ we obtain%
\begin{equation*}
u\geq h\phi \left( -\frac{G^{\Omega }\left( h^{q}V_{+}\right) }{h}\right)
\quad \text{in}\,\,\Omega .
\end{equation*}%
Since $G^{\Omega }\left( h^{q}V_{+}\right) <\infty $, it follows that $u>0$
in $\Omega $ as was claimed.

Let us return to (\ref{eps-v}). Since $v_{\varepsilon }\in I_{q}$ and,
hence, the right hand side of (\ref{eps-v}) lies in $\overline{I}_{q}$, we
can apply $\phi $ to the both sides of this inequality and obtain%
\begin{equation}
u_{\varepsilon }\geq h\phi \left( -\frac{G^{\Omega }\left( h^{q}\left( \frac{%
u}{u_{\varepsilon }}\right) ^{q}V\right) }{h}\right) \ \ \text{in }\Omega ,
\label{eps}
\end{equation}%
The positivity of $u$ in $\Omega $ implies $\frac{u}{u_{\varepsilon }}%
\uparrow 1$ in $\Omega $ as $\varepsilon \rightarrow 0$, whence by the
monotone convergence theorem, 
\begin{equation}
G^{\Omega }\left( h^{q}\left( \frac{u}{u_{\varepsilon }}\right) ^{q}V\right)
\rightarrow G^{\Omega }\left( h^{q}V\right) \ \ \text{as }\varepsilon
\rightarrow 0  \label{mon-conv}
\end{equation}%
pointwise in $\Omega $. In particular, we have, for any $x\in \Omega $, 
\begin{equation}
-\frac{G^{\Omega }\left( h^{q}V\right) \left( x\right) }{h\left( x\right) }%
\in \overline{I}_{q}.  \label{G}
\end{equation}%
Letting $\varepsilon \rightarrow 0$ in (\ref{eps}), we deduce, for $q\geq 1$%
, 
\begin{equation*}
u\geq h\phi \left( -\frac{G^{\Omega }\left( h^{q}V\right) }{h}\right) \ \ 
\text{in }\Omega \text{,}
\end{equation*}%
which proves (\ref{main-exp}) and (\ref{main1}). In the case $q>1$, it
follows that%
\begin{equation*}
\phi \left( -\frac{G^{\Omega }\left( h^{q}V\right) }{h}\right) \leq \frac{u}{%
h}<\infty
\end{equation*}%
and, hence,%
\begin{equation*}
-\frac{G^{\Omega }\left( h^{q}V\right) }{h}<\frac{1}{q-1},
\end{equation*}%
which proves (\ref{main-cond}).

Assume now $0<q<1$. We employ the same argument up to (\ref{eps-v}). The
extended function $\phi $ is defined in this case on $\left[ -\infty
,+\infty \right] $ by (\ref{fi+}). Applying $\phi $ to the both sides of (%
\ref{eps-v}) we obtain%
\begin{equation}
u_{\varepsilon }\geq h\phi \left( -\frac{1}{h}G^{\Omega }\left( h^{q}\left( 
\frac{u}{u_{\varepsilon }}\right) ^{q}V\right) \right) .  \label{uee}
\end{equation}
In this case $u$ can actually vanish inside $\Omega $. Letting $\varepsilon
\rightarrow 0$, we see that $\frac{u}{u_{\varepsilon }}\left( x\right)
\uparrow 1\ $if$\ u\left( x\right) >0$ and $\frac{u}{u_{\varepsilon }}=0\ $%
if $u\left( x\right) =0$, that is 
\begin{equation*}
\frac{u}{u_{\varepsilon }}\uparrow \chi _{u}\ \ \text{pointwise in }\Omega .
\end{equation*}%
Passing to the limit in (\ref{uee}) as $\varepsilon \rightarrow 0$ and using
the monotone convergence theorem gives 
\begin{equation}
u\geq h\phi \left( -\frac{1}{h}G^{\Omega }\left( \chi _{u}h^{q}V\right)
\right) \ \text{in }\Omega .  \label{huu}
\end{equation}%
which is equivalent to (\ref{main}).

Consider the last case $q<0$. We define for any $\varepsilon >0$ the
function $v_{\varepsilon }$ in a slightly different way as follows:%
\begin{equation*}
v_{\varepsilon }=\phi ^{-1}\left( \frac{u}{h_{\varepsilon }}\right) ,
\end{equation*}%
where $h_{\varepsilon }=h+\varepsilon $. Since $\frac{u}{h_{\varepsilon }}>0$
in $\Omega $, we obtain $v_{\varepsilon }\in C^{2}\left( \Omega \right) $.
The function%
\begin{equation}
\phi ^{-1}\left( s\right) =\frac{s^{1-q}-1}{1-q},  \label{fi-1}
\end{equation}%
initially defined for $s>0$, extends continuously to $s=0$ by setting $\phi
^{-1}\left( 0\right) =-\frac{1}{1-q}$. Since $\frac{u}{h_{\varepsilon }}$ is
continuous and nonnegative in $\overline{\Omega }$, we obtain $%
v_{\varepsilon }\in C\left( \overline{\Omega }\right) $. Moreover, since on
the boundary $\partial \Omega $ we have $u\leq h<h_{\varepsilon }$, it
follows that $v_{\varepsilon }\leq \phi ^{-1}\left( 1\right) =0$ and, hence, 
\begin{equation}
h_{\varepsilon }v_{\varepsilon }\leq 0\ \ \text{on\ \ }\partial \Omega .
\label{hee}
\end{equation}%
Since $\mathcal{L}h_{\varepsilon }\leq 0$ and $u=h_{\varepsilon }\phi \left(
v_{\varepsilon }\right) $ satisfies by (\ref{semilinearB}) 
\begin{equation*}
-\mathcal{L}u+Vu^{q}\leq -\mathcal{L}h_{\varepsilon },
\end{equation*}%
we obtain by inequality (\ref{Lh-}) of Lemma \ref{Cor} and (\ref{fi'}) that%
\begin{equation}
-\mathcal{L}(h_{\varepsilon }v_{\varepsilon })+h_{\varepsilon }^{q}\,V\leq
0\ \text{in }\Omega .  \label{Lhe}
\end{equation}%
Since $q<0$ and 
\begin{equation*}
G^{\Omega }(h_{\varepsilon }^{q}V_{\pm })\leq G^{\Omega }(h^{q}V_{\pm }),
\end{equation*}%
it follows that $G^{\Omega }(h_{\varepsilon }^{q}V)$ is well-defined. Hence,
we obtain from (\ref{Lhe}) and (\ref{hee}) by the maximum principle of Lemma %
\ref{Lemmax}, that 
\begin{equation*}
h_{\varepsilon }v_{\varepsilon }\leq -G^{\Omega }(h_{\varepsilon
}^{q}V)\quad \text{in}\,\,\Omega ,
\end{equation*}%
that is,%
\begin{equation}
v_{\varepsilon }\leq -\frac{G^{\Omega }(h_{\varepsilon }^{q}V)}{%
h_{\varepsilon }}\ \ \text{in }\Omega \text{.}  \label{est-q-neg2}
\end{equation}%
Since $v_{\varepsilon }\left( \Omega \right) \subset I_{q}=(-\frac{1}{1-q}%
,\infty )$, it follows that%
\begin{equation}
-\frac{G^{\Omega }(h_{\varepsilon }^{q}V)}{h_{\varepsilon }}\in (-\frac{1}{%
1-q},+\infty ]\subset \overline{I}_{q}.  \label{G0}
\end{equation}%
Applying $\phi $ to both sides of (\ref{est-q-neg2}), we obtain%
\begin{equation*}
\phi \left( v_{\varepsilon }\right) \leq \phi \left( -\frac{G^{\Omega
}(h_{\varepsilon }^{q}V)}{h_{\varepsilon }}\right) \ \ \text{in}\,\,\Omega ,
\end{equation*}%
which is equivalent to 
\begin{equation*}
u\leq h_{\varepsilon }\left[ 1-(1-q)\frac{G^{\Omega }(h_{\varepsilon }^{q}V)%
}{h_{\varepsilon }}\right] ^{\frac{1}{1-q}}\ \ \text{in }\Omega
\end{equation*}%
and, hence, to%
\begin{equation}
u\leq h_{\varepsilon }\left[ 1-(1-q)\,\frac{G^{\Omega }(h_{\varepsilon
}^{q}V_{+})}{h_{\varepsilon }}+(1-q)\,\frac{G^{\Omega }(h_{\varepsilon
}^{q}V_{-})}{h_{\varepsilon }}\right] ^{\frac{1}{1-q}}.  \label{uhe}
\end{equation}%
Note that the expression in the square brackets here belongs to $(0,+\infty
] $ by (\ref{G0}). In particular, we have $G^{\Omega }(h_{\varepsilon
}^{q}V_{+})<\infty $. Since $0<h<h_{\varepsilon }$ in $\Omega $ and $q<0$,
we see that in $\Omega $ 
\begin{equation}
\frac{G^{\Omega }(h_{\varepsilon }^{q}V_{-})}{h_{\varepsilon }}\leq \frac{%
G^{\Omega }(h^{q}V_{-})}{h}.  \label{VV}
\end{equation}
Since $h_{\varepsilon }^{q}\uparrow h^{q}$ as $\varepsilon \rightarrow 0$,
we obtain by the monotone convergence theorem, that 
\begin{equation}
G^{\Omega }(h_{\varepsilon }^{q}V_{+})\rightarrow G^{\Omega }(h^{q}V_{+})\ \ 
\text{pointwise in }\Omega \text{.}  \label{GG}
\end{equation}%
Since by hypothesis $G^{\Omega }(h^{q}V)$ is well-defined, we obtain as $%
\varepsilon \rightarrow 0$ from (\ref{uhe}), (\ref{VV}) and (\ref{GG}) that 
\begin{equation*}
u\leq h\left[ 1-(1-q)\,\frac{G^{\Omega }(h^{q}V)}{h}\right] ^{\frac{1}{1-q}%
}\ \text{in\ \ }\Omega .
\end{equation*}%
By construction the expression in the square brackets here belongs to $\left[
0,+\infty \right] $. Since by hypothesis $u>0$ in $\Omega $, we obtain that
this expression cannot vanish, which proves (\ref{main-cond}) in this case.
\end{proof}

\section{Proof of Theorem \protect\ref{T1}}

\label{SecT1}Consider first the case $q>0$. By hypothesis, the function $f$
is continuous and non-negative in $\Omega $. In the proof we need $f$ to be
locally H\"{o}lder continuous because in this case the function $G^{U}f$ is
of the class $C^{2}$ for any relatively compact domain $U\subset \Omega $.

Let us approximate a given continuous function $f$ in $\Omega $ from below
by a sequence $\left\{ f_{k}\right\} _{k=1}^{\infty }$ of $C^{1}$ functions $%
f_{k}$ so that 
\begin{equation}
f_{k}\uparrow f\ \ \text{as\ \ }k\rightarrow \infty  \label{muk}
\end{equation}%
pointwise. Replacing each $f_{k}$ by $\left( f_{k}\right) _{+}$, we obtain a
sequence $\left\{ f_{k}\right\} $ of nonnegative locally Lipschitz functions
satisfying (\ref{muk}).

Set $h_{k}=G^{\Omega }f_{k}$ and observe that $h_{k}\leq h<\infty $ and $%
h_{k}\uparrow h$ pointwise in $\Omega $ as $k\rightarrow \infty $. Since 
\begin{equation*}
G^{\Omega }\left( h_{k}^{q}V_{\pm }\right) \leq G^{\Omega }\left(
h^{q}V_{\pm }\right) ,
\end{equation*}%
we see that one of the values $G^{\Omega }\left( h_{k}^{q}V_{\pm }\right) $
is finite and, hence, $G^{\Omega }\left( h_{k}^{q}V\right) $ is
well-defined. Since 
\begin{equation*}
G^{\Omega }\left( h_{k}^{q}V_{\pm }\right) \rightarrow G^{\Omega }\left(
h^{q}V_{\pm }\right) ,
\end{equation*}%
we obtain that 
\begin{equation}
G^{\Omega }\left( h_{k}^{q}V\right) \rightarrow G^{\Omega }\left(
h^{q}V\right)  \label{Ghk}
\end{equation}%
pointwise in $\Omega $. The same is true for $G^{\Omega }\left( \chi
_{u}h_{k}^{q}V\right) $ in the case $\left( iii\right) $.

Since $f_{k}\leq f$, we obtain that $u$ satisfies $-\mathcal{L}u+Vu^{q}\geq
f_{k}$ in $\Omega $. Therefore, if statements $\left( i\right) ,\left(
ii\right) ,\left( iii\right) $ are already proved for locally Lipschitz
functions $f$, then we obtain the corresponding lower bounds (\ref{main-exp}%
), (\ref{main1}), (\ref{main}) of $u$ with $h_{k}$ in place of $h$. Letting $%
k\rightarrow \infty $ and using (\ref{Ghk}), we obtain the same estimates of 
$u$ via $h$ as claimed.

In the case $\left( ii\right) $ we still need to prove (\ref{main-cond}) for 
$h$ assuming that it is true with $h_{k}$ in place of $h$. Passing to the
limit as $k\rightarrow \infty $, we obtain a non-strict inequality%
\begin{equation}
-(q-1)G^{\Omega }(h^{q}V)(x)\leq h(x).  \label{mc}
\end{equation}%
However, estimate (\ref{main1}) implies that the expression in the square
brackets in (\ref{main1}) cannot vanish, which yields a strict inequality in
(\ref{mc}), that is, (\ref{main-cond}).

Continuing the proof in the case $q>0$, we can assume now that $f$ is
locally H\"{o}lder (even Lipschitz) continuous. Let $\left\{ \Omega
_{n}\right\} _{n=1}^{\infty }$ be an exhaustion of $\Omega $ by relatively
compact, connected, open sets $\Omega _{n}\Subset \Omega $ with smooth
boundaries. Set $h_{n}=G^{\Omega _{n}}f$. Since $f$ is locally H\"{o}lder
continuous and $\partial \Omega _{n}$ is regular, we have $h_{n}\in
C^{2}\left( \Omega _{n}\right) \cap C\left( \overline{\Omega }_{n}\right) $%
\textbf{\ }and%
\begin{equation*}
\begin{cases}
-\mathcal{L}h_{n}=f\quad & \text{in}\,\,\Omega _{n}, \\ 
h_{n}=0\quad & \text{on}\,\,\partial \Omega _{n}.%
\end{cases}%
\end{equation*}%
We can always take $n$ large enough so that $f\not\equiv 0$ in $\Omega _{n}$
and, hence, $0<h_{n}<\infty $ in $\Omega _{n}$.

Observe that by the monotone convergence theorem 
\begin{equation*}
h_{n}\uparrow h:=G^{\Omega }f\ \ \text{as }n\rightarrow \infty .
\end{equation*}%
Fix a point $x\in \Omega $ and let $n$ be so large that $x\in \Omega _{n}$.
Since $u$ satisfies (\ref{semilinear2A}) in $\Omega $, it follows that 
\begin{equation*}
\begin{cases}
-\mathcal{L}u+Vu^{q}\geq f=-\mathcal{L}h_{n}\quad & \text{in}\,\,\Omega _{n},
\\ 
u\geq 0=h_{n}\ \quad & \text{on}\,\,\partial \Omega _{n}.%
\end{cases}%
\end{equation*}%
Applying Theorem~\ref{T2} in $\Omega _{n}$ we obtain%
\begin{equation}
u(x)\geq 
\begin{cases}
\,\,h_{n}(x)e^{-\frac{G^{\Omega _{n}}(h_{n}V)(x)}{h_{n}(x)}},\quad & \text{if%
}\,\,q=1, \\ 
\,\,h_{n}(x)\left[ 1+(q-1)\frac{G^{\Omega _{n}}(h_{n}^{q}V)(x)}{h_{n}(x)}%
\right] ^{-\frac{1}{q-1}},\quad & \text{if}\,\,q>1, \\ 
h_{n}(x)\left[ 1+(q-1)\frac{G^{\Omega _{n}}(\chi _{{n}}h_{n}^{q}V)(x)}{%
h_{n}(x)}\right] _{+}^{-\frac{1}{q-1}}, & \text{if\ }0<q<1,%
\end{cases}
\label{omA}
\end{equation}%
where $\chi _{n}:=\chi _{u|_{\Omega _{n}}}$. Since $h_{n}^{q}\uparrow h^{q}$
as $n\rightarrow \infty $, we obtain by the monotone convergence theorem, 
\begin{equation}
\lim_{n\rightarrow \infty }G^{\Omega _{n}}(h_{n}^{q}V_{\pm })(x)=G^{\Omega
}(h^{q}V_{\pm })(x)  \label{mon-convthm}
\end{equation}%
(and a similar identity for the term with $\chi _{{n}}h_{n}^{q}V$). Passing
to the limit in (\ref{omA}) as $n\rightarrow \infty $, we arrive at 
\begin{equation}
u(x)\geq 
\begin{cases}
\,\,h(x)e^{-\frac{G^{\Omega }(hV)(x)}{h(x)}},\quad & \text{if}\,\,q=1, \\ 
\,\,h(x)\left[ 1+(q-1)\frac{G^{\Omega }(h^{q}V)(x)}{h(x)}\right] ^{-\frac{1}{%
q-1}},\quad & \text{if}\,\,q>1, \\ 
\,h(x)\left[ 1+(q-1)\frac{G^{\Omega }(\chi _{u}h^{q}V)(x)}{h(x)}\right]
_{+}^{-\frac{1}{q-1}}, & \text{if }0<q<1,%
\end{cases}
\label{omC}
\end{equation}%
which proves estimates (\ref{main1}), (\ref{main}), (\ref{main-negative}).

In the case $q>1$ the expression in square brackets in (\ref{omC}) is
non-negative as the limit of that of (\ref{omA}). However, since the
exponent $-\frac{1}{q-1}$ is in this case negative and $\frac{u\left(
x\right)}{h\left( x\right)} <\infty $, it actually has to be positive, which
proves (\ref{main-cond}).

Consider now the case $q<0.$ In this case we approximate $f$ from above by a
sequence of $C^{1}$ functions $f_{k}$ such that $f_{k}\downarrow f$ and set $%
h_{k}=G^{\Omega }f_{k}$. The function $f_{1}$ should be chosen so close to $%
f $ that $h_{1}<\infty $. Then $h_{k}\downarrow h$ pointwise in $\Omega $,
and, since $q<0$, we have $h_{k}^{q}\uparrow h^{q}$ as $k\rightarrow \infty $%
. The same argument as in the case $q>0$ shows that $G^{\Omega }\left(
h_{k}^{q}V\right) $ is well-defined and (\ref{Ghk}) holds. Since $f_{k}\geq
f $, the function $u$ satisfies in $\Omega $ the inequality $-\mathcal{L}%
u+Vu^{q}\leq f_{k}$. If $\left( iv\right) $ is already proved for locally H%
\"{o}lder continuous $f$, then we conclude that (\ref{main-negative}) holds
with $h_{k}$ instead of $h$. Letting $k\rightarrow \infty $, we complete the
proof (condition (\ref{main-cond}) is proved in the same way as in the case $%
q>0$).

Hence, we assume in what follows that $f$ is locally H\"{o}lder continuous.
In this case the proof goes the same way as in Theorem \ref{T2}. Observe
first that $G^{\Omega }f\in C^{2}\left( \Omega \right) $. Indeed, for any
relatively compact open set $\Omega ^{\prime }\subset \Omega $ with smooth
boundary it is known that $G^{\Omega ^{\prime }}f\in C^{2}\left( \Omega
^{\prime }\right) $. Since the difference $G^{\Omega }f-G^{\Omega ^{\prime
}}f$ is harmonic in $\Omega ^{\prime }$, it follows that it is smooth in $%
\Omega ^{\prime }$, which implies that $G^{\Omega }f\in C^{2}\left( \Omega
^{\prime }\right) $. By exhausting $\Omega $ with relatively compact open
subsets, we obtain $G^{\Omega }f\in C^{2}\left( \Omega \right) $ as claimed.

For any $\varepsilon >0$ set $h_{\varepsilon }=\varepsilon +G^{\Omega }f,$
so that $-\mathcal{L}h_{\varepsilon }=f$. Since $u,h_{\varepsilon }>0$ in $%
\Omega $, the function $v_{\varepsilon }=\phi ^{-1}\left( \frac{u}{%
h_{\varepsilon }}\right) $ belongs to $C^{2}\left( \Omega \right) $ and,
similarly to the proof of Theorem \ref{T2} (cf. (\ref{Lhe})), we obtain the
following inequality in $\Omega $%
\begin{equation*}
-\mathcal{L}\left( h_{\varepsilon }v_{\varepsilon }\right) +h_{\varepsilon
}^{q}V\leq 0.
\end{equation*}%
Note that in this case we have by (\ref{fi-1})%
\begin{equation*}
h_{\varepsilon }v_{\varepsilon }=h_{\varepsilon }\phi ^{-1}\left( \frac{u}{%
h_{\varepsilon }}\right) =h_{\varepsilon }^{q}\frac{u^{1-q}-h_{\varepsilon
}^{1-q}}{1-q}.
\end{equation*}%
Using the boundary condition in (\ref{semilinear2B}) and $h_{\varepsilon
}\geq \varepsilon $, we obtain%
\begin{equation*}
\limsup_{y\rightarrow \partial _{\infty }\Omega }\left( h_{\varepsilon
}v_{\varepsilon }\right) \left( y\right) \leq 0.
\end{equation*}%
Applying Lemma \ref{Lemmax} to $-h_{\varepsilon }v_{\varepsilon }$ we obtain%
\begin{equation*}
-h_{\varepsilon }v_{\varepsilon }\geq G^{\Omega }\left( h_{\varepsilon
}^{q}V\right) .
\end{equation*}%
Letting $\varepsilon \rightarrow 0$ and arguing as in the proof of Theorem %
\ref{T2}, we finish the proof.

\begin{remark}
\RM Note that (\ref{omA}) implies immediately the lower bounds of Theorem %
\ref{T1}$\left( i\right)$, $\left( ii\right)$, $\left( iii\right) $ by
passing to the limit as $n\rightarrow \infty $, provided we use a relaxed
definition of the expression $G^{\Omega }\left( h^{q}V\right) $ given by (%
\ref{improper}). A similar observation holds also for the upper estimate of $%
\left( iv\right) $.
\end{remark}

\section{Proof of Theorem \protect\ref{T3}}

\label{SecT3}The proof is similar to that of Theorem~\ref{T2}, but simpler.
Let $\left\{ \Omega _{n}\right\} $ be an exhaustion of $\Omega $ as above.

Assume first $q\geq 1$ and define for any $n$ a function $h_{n}\in
C^{2}\left( \Omega _{n}\right) \cap C\left( \overline{\Omega }\right) $ as
the solution of 
\begin{equation*}
\left\{ 
\begin{array}{ll}
\mathcal{L}h_{n}=0 & \text{in\ }\Omega _{n}, \\ 
h_{n}=u & \text{on }\partial \Omega _{n}.%
\end{array}%
\right.
\end{equation*}%
In cases $\left( i\right), \left( ii\right) $, we have $h_{n}>0$ in $\Omega
_{n}$ for large enough $n$ by (\ref{u=1-thm3}) and (\ref{u=infty-thm3})
respectively. \textbf{\ }By Theorem~\ref{T2} it follows that $u(x)>0$ for
all $x\in \Omega _{n}$. Consequently, $u(x)>0$ for all $x\in \Omega $.

In the case $q=1$, set $h\equiv 1$, $v=\ln u$. As in the proof of Theorem~%
\ref{T2} (cf. (\ref{ineq})), we obtain 
\begin{equation*}
-\mathcal{L}v+V\geq 0.
\end{equation*}%
Since by (\ref{u=1-thm3}) we have \thinspace $\liminf_{y\rightarrow \partial
_{\infty }\Omega }v\left( y\right) \geq 0$, we conclude by Lemma \ref{Lemmax}
that 
\begin{equation}
\ln u(x)=v(x)\geq -G^{\Omega }V(x),  \label{omega+}
\end{equation}%
which proves (\ref{est1-thm3}).

In the case $q>1$, we set $\nu _{n}=\inf_{\partial \Omega _{n}}u$, where by (%
\ref{u=infty-thm3}) we can assume $\lim_{n\rightarrow \infty }\nu
_{n}=+\infty $. Then by Theorem~\ref{T2} with $h\equiv \nu _{n}$, we obtain
in $\Omega _{n}$ 
\begin{eqnarray}
u &\geq &\nu _{n}\left[ 1+(q-1)\nu _{n}^{q-1}G^{\Omega _{n}}V\right] ^{-%
\frac{1}{q-1}}  \notag \\
&=&\left[ \nu _{n}^{-\left( q-1\right) }+(q-1)G^{\Omega _{n}}V\right] ^{-%
\frac{1}{q-1}}  \label{est5-thm3}
\end{eqnarray}%
where 
\begin{equation}
-(q-1)G^{\Omega _{n}}V<\nu _{n}^{-\left( q-1\right) }\quad \mathrm{in}%
\,\,\Omega _{n}.  \label{-q}
\end{equation}%
It follows from (\ref{-q}) that $G^{\Omega }V_{-}(x)\not=+\infty $, since
otherwise both $G^{\Omega }V_{\pm }(x)=+\infty $. Hence, by letting $%
n\rightarrow +\infty $ in (\ref{-q}), we see that $G^{\Omega }V\left(
x\right) \geq 0,$ and consequently by the monotone convergence theorem (\ref%
{est5-thm3}) yields 
\begin{equation*}
u\left( x\right) \geq \left[ (q-1)G^{\Omega }V\left( x\right) \right] ^{-%
\frac{1}{q-1}}.
\end{equation*}%
Since $u(x)<\infty $, we actually have a strict inequality $G^{\Omega
}V(x)>0 $.

Consider now the case $0<q<1$. We set 
\begin{equation*}
\phi (v)=\left[ (1-q)v\right] ^{\frac{1}{1-q}},\quad v\in I_{q}=(0,+\infty ).
\end{equation*}%
Then clearly 
\begin{equation*}
\phi ^{\prime }(v)=\left[ (1-q)v\right] ^{\frac{q}{1-q}}>0,\quad \phi
^{\prime \prime }(v)=q\left[ (1-q)v\right] ^{\frac{2q-1}{1-q}}>0,
\end{equation*}%
and (\ref{fi'}) holds. For a sequence $\varepsilon _{n}\downarrow 0$, we set 
$u_{n}=u+\varepsilon _{n}$, and define $v_{n}$ by 
\begin{equation*}
v_{n}=\phi ^{-1}(u_{n}),\quad n=1,2,\ldots .
\end{equation*}%
Using Lemma~\ref{Cor} in the case $h\equiv 1$ so that $\mathcal{L}h=0$ (in
this case the condition $\phi (0)=1$ in (\ref{conv-incr}) is not required,
see Remark after the proof of Lemma \ref{Cor}), we obtain as in the proof of
Theorem~\ref{T2}, 
\begin{equation*}
-\mathcal{L}v_{n}+\left( \frac{u}{u_{n}}\right) ^{q}V\geq 0.
\end{equation*}%
Since $v_{n}>0$ on $\partial \Omega _{n}$, it follows from the maximum
principle%
\begin{equation}
v_{n}\geq -G^{\Omega _{n}}\left( \left( \frac{u}{u_{n}}\right) ^{q}V\right)
\quad \text{in}\,\,\Omega _{n}.  \label{lower-thm3}
\end{equation}%
As $n\rightarrow \infty $ we obtain $v_{n}\rightarrow \phi ^{-1}(u)$, and 
\begin{equation*}
\lim_{n\rightarrow \infty }G^{\Omega _{n}}\left( \left( \frac{u}{u_{n}}%
\right) ^{q}V_{\pm }\right) =G^{\Omega }\left( \chi _{\Omega ^{+}}V_{\pm
}\right)
\end{equation*}%
by the monotone convergence theorem. Passing to the limit in (\ref%
{lower-thm3}) as $n\rightarrow \infty $ gives 
\begin{equation*}
\phi ^{-1}(u)\geq -G^{\Omega }\left( \chi _{\Omega ^{+}}V\right) ,
\end{equation*}%
which is equivalent to (\ref{est2-thm3}).

Finally, let $q<0$. We argue as in the case $q>1$, setting $\nu
_{n}=\inf_{\partial \Omega _{n}}u$ where in view of (\ref{u=0-thm3}) we can
assume $\lim_{n\rightarrow \infty }\nu _{n}=0$. Then by Theorem~\ref{T2}
with $h\equiv \nu _{n}$, 
\begin{equation}
u(x)\leq \left[ \nu _{n}^{1-q} - (1-q)G^{\Omega _{n}}V (x)\right] ^{\frac{1}{%
1-q}}\quad \mathrm{in}\,\,\Omega _{n},  \label{est7-thm3}
\end{equation}%
where 
\begin{equation*}
(1-q)\,G^{\Omega _{n}}V(x)<\nu _{n}^{1-q}\quad \mathrm{in}\,\,\Omega _{n}.
\end{equation*}%
It follows as in the case $q>1$ that $GV_{+}(x)\not=+\infty $, and $%
GV_{+}(x)\leq GV_{-}(x)$. Letting $n\rightarrow +\infty $ in (\ref{est7-thm3}%
), we deduce (\ref{est4-thm3}), which yields the strict inequality $%
GV_{+}(x)<GV_{-}(x)$, since $u(x)>0$.

\section{Proof of Theorems \protect\ref{C4} and \protect\ref{T4}}

\label{SecProofE}

\begin{proof}[Proof of Theorem~\protect\ref{C4}]
We prove only statement $(ii)$ (for $q<0$) since statement $(i)$ (for $q>1$)
is proved in a similar but simpler way. We use the method of sub- and
super-solutions, understood in the classical sense: if there exist $%
\underline{u}$, $\overline{u}\in C(\overline{\Omega })\cap C^{2}(\Omega )$
such that $0<\underline{u}\leq \overline{u}$ in $\Omega $, $\underline{u}=%
\overline{u}=0$ on $\partial \Omega $, and 
\begin{equation*}
-\mathcal{L}\underline{u}+V\underline{u}^{q}\leq f,\qquad -\mathcal{L}%
\overline{u}+V\overline{u}^{q}\geq f\ \ \text{in}\,\,\Omega ,
\end{equation*}%
then there exists a solution $u\in C(\overline{\Omega })\cap C^{2}(\Omega )$
to (\ref{diff}) such that $\underline{u}\leq u\leq \overline{u}$. (See \cite%
{DGR}, Theorem 1.2.3, in the case $M=\mathbb{R}^{n}$ and $\mathcal{L}=\Delta 
$; the same proof which relies on standard interior regularity estimates
works in the general case.)

Clearly, setting $\overline{u}=h=G^{\Omega }f\in C(\overline{\Omega })\cap
C^{2}(\Omega )$ gives a supersolution since $V\geq 0$, and consequently 
\begin{equation*}
-\mathcal{L}\overline{u}+V\overline{u}^{q}\geq -\mathcal{L}\overline{u}%
=f,\quad \overline{u}=0\,\,\text{on}\,\,\partial \Omega .
\end{equation*}

The main problem is to find a subsolution which we define by 
\begin{equation*}
\underline{u}=h-\lambda ^{q}\,G^{\Omega }(h^{q}V),
\end{equation*}%
where $\lambda >0$ is a constant to be determined later. Using (\ref%
{cond2-corr}) we see that $\underline{u}>0$ provided 
\begin{equation}
\left( 1-\frac{1}{q}\right) ^{q}\frac{1}{1-q}<\lambda ^{-q}.
\label{lambda-C4}
\end{equation}

Under the assumptions imposed on $f$ it follows that $h\in C(\overline{%
\Omega })\cap C^{2}(\Omega )$. We need to show that $\underline{u}\in C(%
\overline{\Omega })\cap C^{2}(\Omega )$. As in the proof of Theorem~\ref{T1} 
$(iv)$, let $\Omega ^{\prime }$ be an arbitrary relatively compact subset of 
$\Omega $ with smooth boundary. Then $G^{\Omega }(h^{q}V)-G^{\Omega^{\prime
}}(h^{q}V)$ is a harmonic function in $\Omega^{\prime }$. Since $h>0$ in $%
\Omega^{\prime }$, it follows that $h^{q}V\in C(\overline{\Omega })$ and is
locally H\"{o}lder-continuous. Hence, $G^{\Omega^{\prime }}(h^{q}V)\in
C^{2}(\Omega ^{\prime })$, and consequently $G^{\Omega }(h^{q}V)\in
C^{2}(\Omega ^{\prime })$ as well. To show that $G^{\Omega }(h^{q}V)\in C(%
\overline{\Omega })$, notice that $h$ vanishes continuously on $\partial
\Omega $. Using (\ref{cond2-corr}), we deduce that the same is true for $%
G^{\Omega }(h^{q}V)$.

It remains to show that $-\mathcal{L}\underline{u}+V\underline{u}^{q}\leq f$%
. Since $q<0$ and hence $\underline{u}^{q}\geq h^{q}$, it follows 
\begin{equation*}
-\mathcal{L}\underline{u}+V\underline{u}^{q}=f-\lambda ^{q}\,h^{q}V+%
\underline{u}^{q}V\leq f,
\end{equation*}%
provided 
\begin{equation*}
\lambda \,h\leq \underline{u}=h-\lambda ^{q}\,G^{\Omega }(h^{q}V),
\end{equation*}%
or equivalently, 
\begin{equation*}
G^{\Omega }(h^{q}V)\leq \lambda ^{-q}(1-\lambda )h.
\end{equation*}%
Optimizing over all $\lambda \in (0,1)$, we obtain that the maximum of the
right-hand side is obtained for $\lambda =\frac{1}{1-\frac{1}{q}}$, which
coincides with condition (\ref{cond2-corr}), 
\begin{equation*}
G^{\Omega }(h^{q}V)\leq \Big(1-\frac{1}{q}\Big)^{q}\frac{1}{1-q}\,h.
\end{equation*}%
Notice that (\ref{lambda-C4}) obviously holds with this choice of $\lambda $
as well. Thus, $\underline{u}$ is a classical subsolution which is positive
in $\Omega $, and $\underline{u}\leq \overline{u}$ as desired. Consequently,
there exists a classical solution $u$ such that $\underline{u}\leq u\leq 
\overline{u}$. Moreover, 
\begin{equation*}
u\geq \underline{u}=h-\lambda ^{q}G^{\Omega }(h^{q}V)=h-\Big(1-\frac{1}{q}%
\Big)^{-q}G^{\Omega }(h^{q}V)\geq \frac{1}{1-\frac{1}{q}}\,h,
\end{equation*}%
which proves the lower bound for $u$ in (\ref{uh-}).

The upper bound was obtained above in Theorem~\ref{T1}$(iv)$.
\end{proof}

\begin{proof}[Proof of Theorem~\protect\ref{T4}]
The case $q>1$, $V\leq 0$ is considered in \cite{KV} and \cite{BC} (see also 
\cite{V}), so we give a proof only in the case $q<0$, $V\geq 0$. Let us
assume that 
\begin{equation}
K(h^{q}V)(x)\leq a\,h(x)\quad dm-\mathrm{a.e.}\,\,\mathrm{in}\,\,\Omega ,
\label{condV++}
\end{equation}%
for some constant $a>0$, where $h$ satisfies (\ref{cond-h}).

Set $u_{0}=h$, and construct a sequence of consecutive iterations $u_{k}$ by 
\begin{equation*}
u_{k+1}+K(u_{k}^{q}V)=h,\quad k=0,1,2,\ldots .
\end{equation*}%
Clearly, by (\ref{condV++}), 
\begin{equation*}
(1-a)h(x)\leq u_{1}(x)=h(x)-K(h^{q}V)(x)\leq h(x)=u_{0}(x).
\end{equation*}%
We set $b_{0}=1$, $b_{1}=1-a$, and continue the argument by induction.
Suppose that for some $k=1,2,\ldots $ 
\begin{equation}
b_{k}\,h(x)\leq u_{k}(x)\leq u_{k-1}(x)\quad \mathrm{in}\,\,\Omega .
\label{induct}
\end{equation}%
Since $q<0$ and $V\geq 0$, we deduce using estimates (\ref{condV++}) and (%
\ref{induct}), 
\begin{equation*}
(1-a\,b_{k}^{q})\,h(x)\leq h(x)-b_{k}^{q}\,K(h^{q}V)(x)\leq
h(x)-K(u_{k}^{q}V)(x)=u_{k+1}(x).
\end{equation*}%
On the other hand, 
\begin{equation*}
u_{k+1}(x)=h(x)-K(u_{k}^{q}V)(x)\leq h(x)-K(u_{k-1}^{q}V)(x)=u_{k}(x).
\end{equation*}%
Hence, 
\begin{equation*}
b_{k+1}\,h(x)\leq u_{k+1}(x)\leq u_{k}(x),\quad \mathrm{where}%
\,\,b_{k+1}=1-a\,b_{k}^{q}.
\end{equation*}%
We need to pick $a>0$ small enough, so that $b_{k}\downarrow b$, where $b>0$%
, and $b=1-a\,b^{q}$.

In other words, we are solving the equation 
\begin{equation}
\frac{1-x}{a}=x^{q}  \label{element}
\end{equation}%
by consecutive iterations $b_{k+1}=1-ab_{k}^{q}$ starting from the initial
value $b_{0}=1$. Clearly, this equation has a solution $0<x<1$ if and only
if $0<a\leq a_{\ast }$, where $y=\frac{1-x}{a_{\ast }}$ is the tangent line
to the convex curve $y=x^{q}$. Here the optimal value $a_{\ast }$ is found
by equating the derivatives, and solving the system of equations 
\begin{equation*}
x_{\ast }^{q}=\frac{1-x_{\ast }}{a},\quad qx_{\ast }^{q-1}=-\frac{1}{a_{\ast
}},
\end{equation*}%
which gives 
\begin{equation*}
a_{\ast }=\Big (1-\frac{1}{q}\Big)^{q}\frac{1}{1-q},\quad x_{\ast }=\frac{1}{%
1-\frac{1}{q}}.
\end{equation*}

Letting $a=a_{\ast }$, we see that by the convexity of $y=x^{q}$, (\ref%
{element}) has a unique solution $x_{\ast }=\frac{1}{1-\frac{1}{q}}$, and by
induction, $x_{\ast }<b_{k+1}<b_{k}<1$, so that 
\begin{equation*}
b_{k}\downarrow b=x_{\ast }=\frac{1}{1-\frac{1}{q}}>0.
\end{equation*}%
%
%
%
%
%
%
%
%
%
%
%
%
%
%
%
%
%
%
%
%
%
%
%
%

From this it follows that (\ref{induct}) holds for all $k=1,2,\ldots $.
Passing to the limit as $k\rightarrow \infty $, and using the monotone
convergence theorem shows that $u=\lim_{k\rightarrow \infty }u_{k}$ is a
solution of (\ref{int}) such that 
\begin{equation*}
b\,h(x)\leq u(x)\leq u_{0}(x)=h(x).
\end{equation*}

Moreover, it is easy to see by construction that $u$ is a maximal solution,
that is, if $\tilde{u}$ is another nonnegative solution to (\ref{int}), then 
$\tilde{u}\leq u_{k}$ for every $k=0,1,2,\ldots $, and consequently $\tilde{u%
}\leq u$ in $\Omega $.
\end{proof}

\section{Examples}

\label{examples}In this section, we consider several examples which
demonstrate various phenomena that may affect behavior of solutions to the
equations considered above. In the liner case $q=1$ (Schr\"{o}dinger
equations), many examples concerning possible behavior of Green's functions
on domains and manifolds for $V\geq 0$ are given in \cite{GH}; the case $%
V\leq 0$ is considered in \cite{FV} and \cite{FNV} (see also \cite{CZ}, \cite%
{GH1}, \cite{Mu}, \cite{P}). In the superlinear case for $q>1$ and $V\geq 0$
we refer to \cite{BC} and \cite{KV} for existence results as well as
pointwise estimates of solutions, and many examples. The case $q>1$ and $%
V\leq 0$ (equations with absorption) is studied in \cite{MV}. In the
sublinear case $0<q<1$, existence of bounded positive solutions, along with
uniqueness, and pointwise estimates of bounded solutions on $\mathbb{R}^{n}$
were obtained in \cite{BK}. Recently, sharp existence results and matching
two-sided estimates for weak positive solutions (not necessarily bounded) in 
$\mathbb{R}^{n}$ were given in \cite{DV1}; see also \cite{DV2}\textbf{\ }for
a characterization of finite energy solutions.

Here we give an example involving a rapidly oscillating $V$ in the case $q=1$%
, and also illustrate various phenomena with regards to pointwise behavior
of solutions in the less studied case $q<0$, for both $V \ge 0$ and $V<0$.
(Related results for $q<0$ where obtained in \cite{AAC}, \cite{DGR}, \cite%
{Gh} and \cite{GhR}.)

\noindent \textbf{Example 1.} We consider first the linear case $q=1$ in
Theorem~\ref{T1}: 
\begin{equation}
-u^{\prime \prime }+V\,u=f\quad \text{in}\,\,\Omega ,  \label{ex-eq1}
\end{equation}%
for $\Omega =(0,1)$, $M=\mathbb{R}^{1}$. Let $f=1$, and $h=Gf=\frac{1}{2}%
x(1-x)$. The corresponding Green function is $G(x,y)=\min (x(1-y),y(1-x))$.

We start with a positive solution with zero boundary values to (\ref{ex-eq1}%
), 
\begin{equation}
u(x)=x(1-x)\Big (1+x\sin \Big(\frac{\pi }{x^{\alpha }}\Big)\Big),\qquad x\in
(0,1),\quad \alpha >0.  \label{ex-eq2}
\end{equation}%
Then 
\begin{equation}
u^{\prime }(x)=x(1-x)\Big (1+x\sin \Big(\frac{\pi }{x^{\alpha }}\Big)\Big)%
,\qquad x\in (0,1),\quad \alpha >0.  \label{ex-eq2a}
\end{equation}%
The corresponding $V=\frac{u^{\prime \prime }+1}{u}$ is found from (\ref%
{ex-eq1}), 
\begin{equation}
V=V_{1}+V_{2}+V_{3},  \label{ex-eq3}
\end{equation}%
where 
\begin{equation*}
\begin{aligned} V_{1}(x) &=-\frac{\alpha ^{2}\pi ^{2}x^{-2\alpha -1}\sin
(\frac{\pi }{x^{\alpha }})}{1+x\sin (\frac{\pi }{x^{\alpha }})}, \\ V_{2}(x)
&=\frac{\alpha (\alpha -1)(1-2x)\pi x^{-\alpha -1}\cos (\frac{\pi
}{x^{\alpha }})-\alpha \pi (1-2x)x^{-\alpha }\cos (\frac{\pi }{x^{\alpha
}})}{1+x\sin (\frac{\pi }{x^{\alpha }})}, \\ V_{3}(x) &=\frac{(1-2x)\sin
(\frac{\pi }{x^{\alpha }})}{1+x\sin (\frac{\pi }{x^{\alpha }})}-\frac{2\sin
(\frac{\pi }{x^{\alpha }})}{(1-x)(1+x\sin (\frac{\pi }{x^{\alpha }}))}.
\end{aligned}
\end{equation*}

Thus, $V$ has a highly oscillatory behavior at the end-point $x=0$, where $%
V_{1}$ is the leading term. Nevertheless, due to the cancellation
phenomenon, we have $u\simeq h$. %
%

For $0<\alpha <1$, $G(hV)$ is well-defined, and Theorem~\ref{T1} gives the
lower bound 
\begin{equation}
u\geq h\,e^{-\frac{G(hV)}{h}},  \label{ex-eq4}
\end{equation}%
which is sharp since $\frac{G(hV)}{h}$ is a bounded function on $\Omega $.
Indeed, it is easy to see that the term $G(hV_{3})$ is harmless since $%
hV_{3} $ is bounded in $\Omega $, and hence $G(hV_{3})\simeq h$ at the
endpoints. To estimate $G(hV_{2})$, notice that $|V_{2}(x)|\leq Cx^{-\alpha
-1}$, and consequently by direct estimates 
\begin{equation}
G(h|V_{2}|)(x)=O(x)\quad \text{as}\,\,x\rightarrow 0^{+}.  \label{ex-eq5}
\end{equation}%
It remains to notice that due to cancellation, for $0<\alpha <1$, 
\begin{equation}
G(hV_{1})(x)=O(x)\quad \text{as}\,\,x\rightarrow 0^{+}  \label{ex-eq6}
\end{equation}%
as well. This can be verified by looking at the asymptotics of the integrals
in the expression 
\begin{equation}
G(hV_{1})(x)=(1-x)\int_{0}^{x}\frac{y^{2}(1-y)}{2}V_{1}(y)dy+x\int_{x}^{1}%
\frac{y(1-y)^{2}}{2}V_{1}(y)dy.  \label{ex-eq7}
\end{equation}%
Clearly, $G(hV_{1})(x)\simeq 1-x$ as $x\rightarrow 1^{-}$. For $0<\alpha <1$%
, it is not difficult to see using integration by parts that $%
G(hV_{1})(x)\simeq x$ as $x\rightarrow 0^{+}$; we omit the details here.

If $\alpha =1$, then $G(hV)$ is not well-defined, and the first term on the
right-hand side of (\ref{ex-eq7}) has to be understood as an improper
integral which asymptotically behaves like $x$ as $x\rightarrow 0^{+}$.
However, the second term actually has an extra logarithmic factor, so that 
\begin{equation*}
G(hV)\simeq \,x\log \Big(\frac{1}{x}\Big)\quad \text{as}\,\,x\rightarrow
0^{+}.
\end{equation*}%
This shows that the lower bound $u(x)\geq h\,e^{-\frac{G(hV)}{h}}$ is not
sharp in this case.

\bigskip

\noindent \textbf{Example 2.} Let $q<0$, and let $\Omega $ be a bounded
domain with smooth boundary in $\mathbb{R}^{n}$. Consider inequality (\ref%
{semilinear2B}) with $\mathcal{L}=\Delta $, $f\equiv 1$, and 
\begin{equation*}
V(x)=\frac{\lambda }{d_{\Omega }(x)^{\beta }},\quad x\in \Omega ,\quad
\lambda >0,\,\,\beta >0,
\end{equation*}%
and the corresponding equation 
\begin{equation}
-\Delta u+\frac{\lambda }{d_{\Omega }(x)^{\beta }}u^{q}=1,\quad u>0\quad 
\mathrm{in}\,\,\Omega .  \label{ex-eq8b}
\end{equation}%
We set 
\begin{equation}
h(x)=G^{\Omega }f(x)\simeq d_{\Omega }(x),\quad x\in \Omega .
\label{ex-eq8a}
\end{equation}

Theorem~\ref{T1} $\left(iv\right)$ gives the following necessary condition, 
\begin{equation}  \label{bound-ex2}
(1-q) \frac{G^\Omega (h^q V)(x)}{h(x)} < 1,
\end{equation}
for the existence of a positive solution $u$ to (\ref{semilinear2B}) with
zero boundary values.

It is easy to see via direct estimates of the Green kernel that, for $\beta
\geq 2+q$, we have $G^{\Omega }(h^{q}V)\equiv +\infty $. For $1+q<\beta <2+q$%
, 
\begin{equation*}
\frac{G^{\Omega}(h^{q}V)(x)}{h(x)}\simeq d_{\Omega }(x)^{1+q-\beta },\quad
x\in \Omega ,
\end{equation*}%
For $\beta =1+q$, we have 
\begin{equation*}
\frac{G^{\Omega}(h^{q}V)(x)}{h(x)}\simeq \log \frac{A}{d_{\Omega }(x)},\quad
x\in \Omega ,
\end{equation*}%
where $A=2\,\mathrm{diam}(\Omega )$. Hence, for $\beta \ge 1+q$, condition (%
\ref{bound-ex2}) fails, and (\ref{semilinear2B}) has no positive solutions $%
u \in C^2(\Omega)\cap C(\overline{\Omega})$ with zero boundary values. This
non-existence result was proved earlier in \cite{DGR}, Theorem 2.1.

In the case $0 < \beta<1+q$, direct estimates give 
\begin{equation}  \label{c lambda}
(1-q) \frac{G^\Omega (h^q V)(x)}{h(x)} \le c \, \lambda,
\end{equation}
where $c=c(\Omega, q, \beta)$ is a positive constant.

Theorem~\ref{T1} $\left(iv\right) $ implies that if (\ref{semilinear2B}) has
a solution $u$ with zero boundary values, then actually (\ref{c lambda})
holds with $c \lambda < 1$, and 
\begin{equation*}
u(x)\leq h(x) \Big[1- (1-q) \frac{G^\Omega (h^q V)(x)}{h(x)} \Big]^{\frac{1}{%
1-q}},\quad x\in \Omega ,
\end{equation*}%
by estimate (\ref{main-negative}).

Moreover, if (\ref{c lambda}) holds with $c\,\lambda \leq \left( 1-\frac{1}{q%
}\right) ^{q}$, then by Theorem \ref{C4} there exists a solution $\widetilde{%
u}$ to (\ref{ex-eq8b}) with zero boundary values which satisfies the lower
bound 
\begin{equation*}
\widetilde{u}(x)\geq \frac{1}{1-\frac{1}{q}}h(x),\quad x\in \Omega .
\end{equation*}%
Hence, $\widetilde{u}(x)\simeq d_{\Omega }(x)$, and our general upper bound (%
\ref{main-negative}) is sharp in this case as well.

In Example 4, we will demonstrate that due to a non-uniqueness phenomenon,
equations of the type (\ref{ex-eq8b}) may have other solutions which violate
the lower bound $u (x)\ge c \, d_{\Omega}(x)$.

\bigskip \noindent \textbf{Example 3.} Let $q<0$, and let $\Omega $ be a
bounded smooth domain in $\mathbb{R}^{n}$. We consider (\ref{semilinear2B})
with $f\equiv 1$, and 
\begin{equation*}
V(x)=-\frac{1}{d_{\Omega }(x)^{\beta }},\qquad \beta >0,
\end{equation*}%
where $d_{\Omega }(x)=\limfunc{dist}(x,\partial \Omega )$, together with the
corresponding equation 
\begin{equation}
-\Delta u-\frac{1}{d_{\Omega }(x)^{\beta }}u^{q}=1,\quad u>0\quad \mathrm{in}%
\,\,\Omega .  \label{ex-eq9}
\end{equation}%
As in the previous example, set 
\begin{equation}
h(x)=G^{\Omega }f(x)\simeq d_{\Omega }(x),\quad x\in \Omega ,  \label{ex-eq8}
\end{equation}%
and $A=2\,\mathrm{diam}(\Omega )$.

Theorem~\ref{T1} $\left( iv\right) $ gives the following upper bounds for
all positive solutions $u\in C^{2}(\Omega )\cap C(\overline{\Omega })$ to (%
\ref{semilinear2B}) with zero boundary values: for all $x\in \Omega $,
\smallskip

$(a)$ $u(x)\leq C\,d_{\Omega }(x)$ if $0<\beta <1+q$;

$(b)$ $u(x)\leq C\,d_{\Omega }(x)\,\log ^{\frac{1}{1-q}}\!\Big(\frac{A}{%
d_{\Omega }(x)}\Big)$ if $\beta =1+q$;

$(c)$ $u(x)\leq C\,d_{\Omega }(x)^{\frac{2-\beta }{1-q}}$ if $1+q<\beta <2+q$%
. \smallskip

The corresponding lower bounds for positive \textit{super-solutions}, not
necessarily with zero boundary values, were established in \cite{Gh},
Proposition 2.6 (see also \cite{DGR}, Theorem 3.5): if $u \in
C^2(\Omega)\cap C(\overline{\Omega})$, and 
\begin{equation}  \label{ex-eq8c}
-\Delta u - \frac{1}{d_\Omega^{\beta}} u^q \ge 0, \quad u>0 \quad \mathrm{in}
\, \, \Omega,
\end{equation}
then, for all $x \in \Omega$, \smallskip

$(a^{\prime })$ $u(x)\geq c\,d_{\Omega }(x)$ if $0<\beta <1+q$;

$(b^{\prime })$ $u(x)\geq c\,d_{\Omega }(x)\,\log ^{\frac{1}{1-q}}\!\Big(%
\frac{A}{d_{\Omega }(x)}\Big)$ if $\beta =1+q$;

$(c^{\prime })$ $u(x)\geq c\,d_{\Omega }(x)^{\frac{2-\beta }{1-q}}$ if $%
1+q<\beta <2$.\smallskip

There are no positive solutions $u$ to (\ref{ex-eq8c}) in the case $\beta\ge
2$. For $0<\beta<2$, there exists a solution $u \in C^2(\Omega)\cap C(%
\overline{\Omega})$ with zero boundary values to equation (\ref{ex-eq9})
which satisfies both the upper and lower bounds given above.

Thus, our general upper bound (\ref{main-negative}) in Theorem~\ref{T1} $%
\left( iv\right) $ is sharp in all cases, except for $2+q \le \beta <2$,
where $G(h^q V)\equiv -\infty$, so that (\ref{main-negative}) becomes
trivial.

\bigskip \noindent \textbf{Example 4.} In this example, we encounter the
non-uniqueness phenomenon for classical solutions with zero boundary
conditions to semilinear equations with negative exponents $q<0$, where
obviously our estimates are not expected to be sharp for all solutions. For
simplicity, we consider the one-dimensional case, although similar examples
are easy to construct in higher dimensions, with coefficients $V$ depending
only on $d_{\Omega }(x)$.

Consider the following semilinear equation: 
\begin{equation}
-u^{\prime \prime }+V\,u^{q}=f\quad \text{in}\,\,\Omega ,  \label{ex4-eq1}
\end{equation}%
for $q<0$, $\Omega =(-1,1)$, with zero boundary conditions $u(\pm 1)=0$. Set 
$f\equiv 1$ and 
\begin{equation*}
h=Gf=\frac{1}{2}(1-x^{2}).
\end{equation*}%
The corresponding Green function is 
\begin{equation*}
G(x,y)=\min \Big((x+1)(1-y),(y+1)(1-x)\Big).
\end{equation*}

Consider a positive solution with zero boundary values to (\ref{ex4-eq1})
given by 
\begin{equation}  \label{ex4-eq2}
u(x) = \lambda \, (1-x^2)^{\gamma}, \qquad x \in (0, 1), \quad \lambda>0,
\quad \gamma>0.
\end{equation}
Then the corresponding $V=\frac{u^{\prime \prime }+1}{u^q}$ is found from (%
\ref{ex4-eq1}), 
\begin{equation*}
V = V_1+V_2 + V_3,
\end{equation*}
where 
\begin{equation*}
\begin{aligned} V_1(x) &= 4 \lambda^{1-q} \, \gamma (\gamma-1)
(1-x^2)^{\gamma-2-\gamma q}, \\ V_2(x) &= - 2 \lambda^{1-q} \, \gamma
(2\gamma-1) (1-x^2)^{\gamma-1-\gamma q},\\ V_3 (x) &=
\lambda^{-q}(1-x^2)^{-\gamma q}. \end{aligned}
\end{equation*}

In the case $\gamma =1$, clearly, $V_1\equiv 0$, and 
\begin{equation*}
V(x)=\lambda ^{-q} (1-2\lambda ) (1-x^{2})^{-q}.
\end{equation*}%
Then 
\begin{equation*}
\frac{G(h^{q}V)(x)}{h(x)}=(2\lambda )^{-q} (1-2\lambda ),\quad x\in \Omega .
\end{equation*}%
Our estimate (\ref{main-negative}) is sharp in both cases, $V\leq 0$ ($%
\lambda \ge \frac{1}{2}$), and $V\geq 0$ ($0<\lambda <\frac{1}{2}$): 
\begin{equation*}
u(x)\leq \frac{1-x^{2}}{2}\Big[1-(1-q)(2\lambda )^{-q}(1-2\lambda )\Big]^{%
\frac{1}{1-q}},
\end{equation*}%
where the constant in square brackets is positive for any choice of $%
\lambda>0$, $q<0$.

In the case $\gamma \not=1$ the situation is more complicated. Clearly, $V_1$
is now the most singular term.

For $\gamma >1$, the behavior of the solution $u$ given by (\ref{ex4-eq2})
at the end-points $x=\pm 1$ is too good to be captured by the upper estimate
(\ref{main-negative}); obviously, it is not sharp for this particular $u$.
On the other hand, notice that $V>0$ if $2\lambda \gamma <1$; for $\gamma>1$%
, it is easy to see by direct estimates that 
\begin{equation}  \label{C-lambda-q}
\frac{G(h^{q}V)(x)}{h(x)}\le C \lambda^{-q}, \quad x\in \Omega .
\end{equation}%
Since there exists a positive solution, Theorem~\ref{T1} ($iv$) implies that
actually (\ref{C-lambda-q}) holds with $C \lambda^{-q}< \frac{1}{1-q}$.

For $1<\gamma <\frac{1}{2\lambda }$, which ensures that $V>0$, every
positive solution $u$ with zero boundary values obviously satisfies the
upper bound 
\begin{equation*}
u\leq h\quad \mathrm{in}\,\,\Omega .
\end{equation*}%
Moreover, if (\ref{C-lambda-q}) holds with $C\lambda ^{-q}\leq \left( 1-%
\frac{1}{q}\right) ^{q}\frac{1}{1-q}$, then by Theorem~\ref{C4} equation (%
\ref{ex4-eq1}) has a solution $\widetilde{u}$ such that $\widetilde{u}\simeq
h$, for which the upper bound (\ref{main-negative}) is indeed sharp.

If $0<\gamma \leq -\frac{q}{1-q}$, then $V$ is too singular at the
end-points, so that $G(h^{q}V)\equiv+\infty$, and (\ref{main-negative})
trivializes.

In the remaining case $-\frac{q}{1-q}<\gamma <1$, it is easy to see that 
\begin{equation*}
\frac{G(h^{q}V)(x)}{h(x)}\simeq (1-x^{2})^{-q+\gamma -\gamma q-1},\quad x\in
\Omega ,
\end{equation*}%
which blows up as $x \to \pm1$. In this case, (\ref{main-negative}) gives $%
u(x)\leq c(1-x^{2})^{\gamma }$, which is again sharp.

\end{document}